\documentclass[12pt,a4paper]{amsart}
\usepackage{amsmath,amssymb}
\usepackage{graphicx}
\usepackage{calligra}
\usepackage{amsfonts}
\newtheorem{Def}{Definition}
\newtheorem{Lma}{Lemma}
\newtheorem*{Lma*A}{Lemma A}
\newtheorem*{Lma*B}{Lemma B}
\newtheorem{Thm}{Theorem}
\newtheorem*{Cor*cpt}{Corollary \ref{cpt}}
\newtheorem*{Cor*ancient}{Corollary \ref{ancient}}
\newtheorem*{Thm*MT}{Theorem \ref{MT}}
\newtheorem*{Thm*MTL}{Theorem \ref{MTL}}
\newtheorem*{Thm*MTLh}{Theorem \ref{MTLh}}
\newtheorem*{Thm*bounded}{Theorem \ref{bounded}}
\newtheorem*{Cor*GRS}{Corollary \ref{GRS}}
\newtheorem*{Thm*decay}{Theorem \ref{decay}}
\newtheorem*{Thm*taming}{Theorem \ref{taming}}
\newtheorem*{Thm*KMW}{Theorem \ref{KMW}}
\newtheorem*{Thm*ShiG}{Theorem \ref{ShiG}}
\newtheorem*{Thm*ShiL}{Theorem \ref{ShiL}}
\newtheorem{Prop}{Proposition}
\newtheorem{Cor}{Corollary}

\newtheorem{Rmk}{Remark}
\newtheorem*{Rmk*}{Remark}
\begin{document}
\title[Shi-type estimates based on Ricci curvature]
{Shi-type estimates of the Ricci flow based on Ricci curvature}
\author{Chih-Wei Chen}
\address{National Center for Theoretical Sciences(NCTS), Taiwan}
\email{BabbageTW@gmail.com}
\dedicatory{Dedicated to Professor G\'erard Besson on the occasion of his 60th birthday}
\date{7 November, 2016; revised 25 January, 2018}

\begin{abstract}

We prove that the magnitude of the derivative of Ricci curvature can be uniformly controlled by the bounds of Ricci curvature and injectivity radius along the Ricci flow. As a consequence, a precise uniform local bound of curvature operator can be constructed from local bounds of Ricci curvature and injectivity radius among all $n$-dimensional Ricci flows. In particular, we show that every Ricci flow with $|Ric|\leq K$ must satisfy
$|Rm|\leq Ct^{-1}$ for all $t\in (0,T]$, where $C$ depends only on the dimension $n$, and 
$T$ depends on $K$ and the injectivity radius $inj_{g(t)}$. 

In the second part of this paper, we discuss the behavior of Ricci curvature and its derivative when the injectivity radius is thoroughly unknown. In particular, another Shi-type estimate for Ricci curvature is derived when the derivative of Ricci curvature is controlled by the derivative of scalar curvature.
\end{abstract}

\maketitle

\section{Introduction}

The Ricci flow on a Riemannian manifold $(M,g_0)$, which was proposed by R. Hamilton in \cite{Hamilton82}, is defined by 
$$
\left\{
\begin{array}{rcl}
\frac{\partial}{\partial t}g(x,t)  &=&-2 Ric_g(x,t)\\
g(x,0)&=& g_0
\end{array}
\right.
$$
and is presumed to be able to improve the Riemannian metric $g_0$. 
Hamilton showed that a Riemannian metric with positive Ricci curvature on a closed $3$-dimensional manifold can be deformed to be rounder and rounder along the Ricci flow. Indeed, by using interpolation techniques, he derived bounds for all derivatives of the curvature tensor and showed that the metrics $g(t)$, after rescaling, converge in $C_{loc}^\infty$-topology to the standard metric on sphere. Later in \cite{Shi89}, W.-X. Shi showed that, for general Ricci flows, all derivatives of curvature are bounded a priori by the bound of the curvature itself. 
Precisely, if a Ricci flow $g(t)$ satisfies $|Rm|\leq K$ for all $t\in[0,\frac{1}{K}]$ on a closed manifold $M^n$, then $|\nabla^l Rm|\leq C_{n,l}Kt^{-l/2}$ for all $t\in(0,\frac{1}{K}]$, where $C_{n,l}$ denotes a constant depending only on $n$ and $l$. 
This estimate, which is called Shi's estimate, even holds locally for complete non-compact Ricci flows. 
That is, there exists a constant $\theta$ depending only on the dimension $n$ such that if a Ricci flow $g(t)$ satisfies $|Rm|\leq K$ on $B_{2r}\times[0,\frac{\theta}{K}]$, then $|\nabla^l Rm|\leq C_{n,l}K(K+r^{-2}+t^{-1})^{l/2}$ on $B_{r}\times (0,\frac{\theta}{K}]$.
Thus, along the Ricci flow, $C^2$-boundedness implies $C^\infty$-boundedness of metrics.

Some other works revealed that merely the Ricci curvature can control the curvature operator in certain circumstances. 
For instance, by blow-up arguments, N. \v{S}e\v{s}um \cite{Sesum05} and L. Ma and L. Cheng \cite{MaCheng10} showed that along the Ricci flow $|Rm|$ maintains finite value as long as $|Ric|$ does. 
For compact manifolds, B. Wang improved the result of \v{S}e\v{s}um by showing that $|Rm|$ is bounded whenever
$Ric$ has a lower bound and the scalar curvature $R$ has certain space-time integral bound \cite{Wang08}.
He also showed that $R$ must blow up whenever $|Rm|$ blows up at $T$ in the order $o((T-t)^{-2})$ \cite[Theorem 1.3]{Wang12}. See also \cite{EndersMullerTopping, Zhang10, LeSesum} for related results.
The estimates of $|Rm|$ in their results depend on the generic behavior of the metrics $g(t)$. 
On the other hand, a classical result due to M. Anderson \cite{Anderson90} says that $|Rm|$ of a Riemannian manifold can be controlled by the bound of $|Ric|$, $|\nabla Ric|$ and the lower bound of injectivity radius.
Therefore, it is natural to ask whether $|\nabla Ric|$, or $|Rm|$, can be controlled by $|Ric|$ along the Ricci flow. 
This is the main theme of this article. 

First, we confirm that a Shi-type estimate for Ricci tensor holds provided that the injectivity radius 
$inj:M\times[0,T]\to\mathbb{R}_+ $ is bounded from below.

\begin{Thm}[Standard version]\label{MTL}
For any $\delta,\eta>0$ and $n\in\mathbb{N}$, there exist positive constants $\alpha$, $C$ and $\rho$ such that 
for any $K>0$ and any smooth Ricci flow $(M^n,g(t))_{t\in[0,T]}$ with $T\geq\frac{\eta}{K}>0$,
if $$|Ric|\leq K\ \ \mbox{ and }\ \ inj\geq \delta K^{-\frac{1}{2}} \ \ \mbox{ on } B_{4\sqrt{T}}(x_0,t) \mbox{ for all } t\in[0,T],$$ 
then 
$$|\nabla Ric| \leq \alpha  \left( KT t^{-1}\right)^{\frac{3}{2}}\mbox{ on } B_{2\sqrt{T}}(x_0,t) \mbox{ for all } t\in(0,T]$$
and
$$|Rm|\leq CKTt^{-1} \mbox{ on } B_{\rho\sqrt{K^{-1}T^{-1}t}}(x_0,t)\mbox{ for all } t\in(0,T].$$ 
\end{Thm}

Recently, B. Kotschwar, O. Munteanu and J. Wang \cite{KotschwarMunteanuWang15} 
improved the aforementioned results of \v{S}e\v{s}um and Wang by a different approach based on Moser's iteration.
They can clarify the dependency of the bound, which eventually involves only $|Ric|$ and the initial bound of curvature operator. 
Moser's iteration has been used before for similar purpose in \cite{Yang8788}, and later in \cite{ DaiWeiYe96} where X. Dai, G. Wei and R. Ye showed that $|Rm|$
can be controlled by initial $|Ric|$ and the initial conjugate radius up to a short time. 
However, all of these bounds do not link to the bound of $|Ric|$ in a clear manner as the bound given by our theorem. 
For instance, when $|Ric|$ is arbitrarily small, our theorem ensures that so is $|Rm|$. 

Furthermore, our proof of Theorem \ref{MTL} can be modified so that the scale of injectivity radius can unhook with the bound of Ricci curvature. Thus, we do not need a huge injective region to get the estimate of $|Rm|$ when $K$ is small. 
In particular, our theorem can be applied to flows with arbitrarily small initial injectivity radius and the resulting bound depends mainly on the growth of $inj$ with respect to time (cf. the parameter $m$ in the next theorem). 

\begin{Thm}[Strong version]\label{MTLh}
For any $\delta,\eta>0$ and $m,n\in\mathbb{N}$, there exist positive constants $\alpha, C,\rho$ such that 
for any $K>0$ and any smooth Ricci flow $(M^n,g(t))_{t\in[0,T]}$ with $T\geq\frac{\eta}{K}>0$,
if $$|Ric|\leq K\ \ \mbox{ and }\ \ inj\geq \delta\cdot \min\{  K^{-\frac{1}{2}}, h^{\frac{1}{2}}(t) \} \ \ \mbox{ on } B_{4\sqrt{T}}(x_0,t) \mbox{ for all } t\in[0,T],$$ 
where $h(t)\leq t$ is any positive function defined on $(0,T)$ such that 
for all $t^*\in (0,T)$, $h(t)\geq m^{-1}h(t^*)$ on $t\in\left[\frac{1}{2}t^*,t^*\right]$, 
then 
$$|\nabla Ric|\leq \alpha  \left( KT h^{-1}\right)^{\frac{3}{2}}\mbox{ on } B_{2\sqrt{T}}(x_0,t) \mbox{ for all } t\in(0,T]$$
and 
$$|Rm|\leq CKTh^{-1} \mbox{ on } B_{\rho\sqrt{K^{-1}T^{-1}h(t)}}(x_0,t)\mbox{ for all } t\in(0,T].$$ 
\end{Thm}

\begin{Rmk*}
The reader may take $h(t)=t$ or $ \frac{T}{\pi} \sin \frac{\pi t}{T}$ to obtain some intuitions. 
The proof of this version can be found in Section 3.
We note that the assumptions do not directly involve any information of $inj$ at the initial time since $\lim_{t\to 0^+}h(t)=0$. 
This does not mean that our theorem can be applied to a Ricci flow with singular initial data. 
The initial metric is required to be at least $C^3$.
\end{Rmk*}

The growth or boundedness of curvature operator are important issues in the study of gradient Ricci solitons. 
Several a priori curvature estimates have been derived before (cf. \cite{MunteanuMTWang11, MunteanuWang14,MunteanuWang15,
ChowLuYang11, CaoCui14, Deruelle14,DengZhu15} among others).
Using our Theorem \ref{MTL}, we derive a new boundedness result for all types of gradient Ricci solitons.

\begin{Thm}\label{bounded}
For any class of complete non-compact $n$-dimensional gradient Ricci solitons of either shrinking or steady or expanding type, if the Ricci curvature is bounded and the injectivity radius is bounded away from zero, then the curvature operator is bounded by a uniform constant.
\end{Thm}

We further derive a compactness theorem which shows that gradient Ricci solitons have compactness property analogue to Einstein manifolds. Note that the limiting soliton might be a trivial soliton. 

\begin{Cor}\label{GRS}
Let $\lambda\in \{\pm \frac{1}{2}, 0 \}$ and $A\in\mathbb{R}$. 
For any sequence of gradient Ricci solitons $(M_k,g_k,f_k,p_k)_{k\in\mathbb{N}}$ satisfying $Ric_{g_k}+Hess(f_k) = \lambda g_k$, if  
$$|Ric|_{g_k}\leq K,\ \ inj_{g_k}\geq I>0 \ \ \mbox{ and }\ \  |\nabla f|_{g_k}(p_k)\leq A,$$ 
then there exists a subsequence converging smoothly to $(M_\infty,g_\infty,f_\infty,p_\infty)$, which satisfies $Ric_{g_\infty}+Hess(f_\infty)= \lambda g_\infty$ with $f_\infty=\lim_{k\to\infty}f_k$.
\end{Cor}

Besides the applications to Ricci solitons, we derive a curvature taming result which can be applied to 
general Ricci flows whose initial curvature might not have a uniform bound. 
We say that the curvature operator is {\it $k$-tamed} by a constant $C$ along a complete non-compact 
Ricci flow if $|Rm|\leq Ct^{-k}$ for all small $t\geq 0$. 
In \cite{HuangTam15}, S. Huang and L.-F. Tam showed that if a Ricci flow $g(t)$ starting from a non-compact K\"ahler manifold is 1-tamed by some constant $C$, then $g(t)$ remains K\"ahler for $t>0$. Moreover, if $C$ is small enough, then the nonnegativity of holomorphic bisectional curvature can also be preserved. Therefore, it is rather important to find criterions for the taming phenomenon.
Note that such estimate is twofold: we would like to have a uniform bound $C$ on a uniform time interval.
For lower dimensional closed manifolds (dimension $n=2$ or $3$), M. Simon \cite[Theorem 2.1]{Simon09} proved that every $(M^n,g)$ satisfying $diam<D$, $Vol>V>0$ and $Ric\geq-Kg$ with sufficiently small $K$ can generate a solution of the Ricci flow which exists up to a maximal time $T=T(D,V)$. Moreover, $|Rm|$ is 1-tamed by a constant depending only on $D$ and $V$. 
In \cite{DaiWeiYe96}, Dai, Wei and Ye showed that every closed $(M^n,g)$ satisfying $|Ric|\leq 1$ and conjugate radius $\geq r_0$  can generate a solution of the Ricci flow which exists up to a maximal time $T=T(n,r_0)$ and $Rm$ is $\frac{1}{2}$-tamed by $C=C(n,r_0)$.
When $|Ric|\leq K$, their theorem holds with $C=C(n,r_0\sqrt{K})$ and $T=T(n,r_0\sqrt{K})$.
For complete non-compact $(M^n,g)$ with $n\geq 3$, G. Xu \cite[Corollary 1.2]{Xu13} showed that, if $Ric\geq-Kg$ and the averaged $L^p$-norm ($p>\frac{n}{2}$) of $Rm$ has a uniform bound $K_1$ for all geodesic balls $B_{r_0}(x)$ with some radius $r_0>0$, then the Ricci flow must exist and $Rm$ is $\frac{n}{2p}$-tamed by a constant $C=C(K,K_1,r_0,n,p)$ up to $t=T(K,K_1,r_0,n,p)$. 

Thanks to the explicit bound in Theorem \ref{MTLh}, we can derive a taming theorem. 

\begin{Thm} \label{taming}
There exists a universal constant $C=C(n)$ such that
for any smooth Ricci flow $(M^n,g(t))$ and any point $x_0\in M$, 
$$|Rm|(x,t)\leq Ct^{-1} \ \mbox{ on } B_{4r}(x_0,t)   \mbox{ for all } t\in (0,K^{-1}],$$ 
and
$$|Rm|(x,t)\leq CK \ \mbox{ on } B_{4r}(x_0,t)   \mbox{ for all } t\in [K^{-1},r^2],$$ 
where $r^2:= \inf \{ \ t>0\ |\  \inf_{B_{8r}(x_0,t)} inj < \sqrt{t} \}$ and $K$ denotes the maximum of $r^{-2}$
and 
$\sup |Ric|$ on $\bigcup_{[0,r^2]}B_{8r}(x_0,t)$.
\end{Thm}

Before Section 5, where more applications of our main theorems are demonstrated, 
we discuss the approach towards our aim via Moser's iteration in Section 4.
In particular, we derive the following theorem 
by using Kotschwar-Munteanu-Wang's $L^p$-estimate \cite[Proposition 1]{KotschwarMunteanuWang15}. 
To abbreviate the notation, we denote the parabolic region emanated from $B_r(x_0,0)$ up to $t=T$ by $\mathcal{P}(r;x_0,T)$,
namely, $\mathcal{P}(r;x_0,T):= \Omega \times(0,T], \mbox{ where }  \Omega\subset M \mbox{ is the topological region defined by } B_r(x_0,0)$.
Furthermore, $\overline{\mathcal{P}(r;x_0,T)}$ is defined to be $\Omega \times[0,T]$.

\begin{Thm}\label{KMW}
Let $(M^n,g(t))_{t\in[0,T]}$ be a smooth solution of the Ricci flow. 
If $$|Ric|\leq K  \mbox{ in } \overline{\mathcal{P}(4r;x_0,T)} \ \mbox{ and }\ 
\inf_{B_{4r} \left( x_0,0 \right)} inj\geq I>0$$
for some $K,I,r > 0$,
then there exists $C$ depending on $K,I,r,T$ and the dimension $n$ such that 
$|Rm|(x_0,t)\leq C$ in $\mathcal{P}(r;x_0,T)\setminus \mathcal{P}(r;x_0,\frac{1}{2}T)$.
\end{Thm}

The second part of this paper consists of Section 6 and 7. There we discuss how to control the derivative of Ricci curvature when lacking information of injectivity radius. Suppose that Ricci curvature is bounded by $K$ and its derivative $|\nabla Ric|$ is controlled by the derivative of scalar curvature $|\nabla R|$, we prove that both $|\nabla Ric|$ and $|\nabla R|$ are bounded for all $t\in(0,\frac{1}{K}]$.
Precisely, we derive 

\begin{Thm}[Global estimate]\label{ShiG}
There exists a constant $C>0$, depending only on $\alpha, \beta$ and $n$ such that
for every $n$-dimensional closed solution $(M^n,g(t))_{t\in[0,T)}$ of the Ricci flow,
if the Ricci curvature and its derivative satisfy that
$|Ric|\leq K$ and $|\nabla Ric|\leq \alpha K t^{\frac{-1}{2}}+\beta|\nabla R|$
for all $t\in[0,\frac{1}{K}]\subset [0,T)$, where $K$ is a positive constant,
then $$|\nabla Ric|^2\leq C K^2t^{-1}$$ for all $t\in(0,\frac{1}{K}]$.
\end{Thm}

\begin{Thm}[Local estimate]\label{ShiL}
There exist positive constants $\theta_0$ and $C$ depending only on $\alpha, \beta, n$ and $\Lambda$ such that
for every solution $(M^n,g(t))_{t\in[0,\theta_0/K]}$ of the Ricci flow,
if $|Rm|\leq \Lambda$ on $B_r(x_0,0)$, 
$|Ric|\leq K$ and $|\nabla Ric|\leq \alpha K \left(\frac{1}{r^2}+\frac{1}{t}+K\right)^{\frac{1}{2}}+\beta|\nabla R|$
on $\overline{\mathcal{P}(r;x_0,t_0)}$ for some $r\leq\sqrt{\theta_0/K}$ and $t_0\leq\theta_0/K$, then
$$|\nabla Ric|^2\leq CK^2\left(\frac{1}{r^2}+\frac{1}{t}+K\right)$$ on $\mathcal{P}(\frac{r}{\sqrt{2}};x_0,t_0) $.
\end{Thm}

We doubt that $|\nabla Ric|$ can be controlled by $|\nabla R|$ for generic solutions of the Ricci flow. However, we believe that it is true for a large variety of solutions including Ricci solitons. 
A related result appeared earlier in a collaborated work of A. Deruelle and the author \cite[Theorem 2.10]{ChenDeruelle15}.
The reader could find more discussions in the last section.  

\ \\
{\bf Acknowledgement.} The main part of this article was done when I visited Ovidiu Munteanu in University of Connecticut in September 2015. I appreciate the hospitality of the university and precious discussions with Ovidiu. Some part of the work was done when I was a doctoral student at l'Institut Fourier. I would like to thank my advisor G\'erard Besson for his encouragement and all kinds of helps. I am grateful to NCTS and TIMS in Taiwan for the constant generous supports.

\section{Global estimate of $Rm$}

Given an $n$-dimensional Riemannian manifold $(M,g)$ and a point $p\in M$, we can find a local chart $(\mathcal{U}, \varphi), \varphi :\mathcal{U}\to \mathbb{R}^n$, such that $p\in \mathcal{U}\subset M$ and $\varphi=(\varphi_1,\dots,\varphi_n)$ consists of harmonic functions, i.e., $\Delta_g \varphi_k =0$ for all $k=1,\dots,n$. Under these coordinates, the Laplacian of a function $f$ which is defined by  
$\Delta_g f= \frac{1}{\sqrt{\det g}}  \partial_i (\sqrt{\det g}\cdot g^{ij} \cdot\partial_j f) $ can be reduced to
$$\Delta_g f=  g^{ij} \partial_i\partial_j f.$$
Moreover, if on a geodesic ball $B_r(p)\subset\mathcal{U}$ one has $|Ric|\leq K$ and $inj\geq I$, then
M. Anderson \cite[Lemma 2.2]{Anderson90} showed that for any $\sigma\in(0,1)$, there exist $\epsilon=\epsilon(K,n,\sigma)$ and $C_0=C_0(\epsilon,I)$ such that harmonic coordinates $\varphi_k$'s exist on $B_{\epsilon I}$
with $$g_{ij}(p)=\delta_{ij}\ \ \mbox{ and }\ \ | g_{ij} |^{'}_{C^{1,\sigma}}\leq C_0 \mbox{ on }B_{\epsilon I},$$
where
$$|g_{ij}|^{'}_{C^{1,\sigma}}:= \sup_{B_r} |g_{ij}| +\sup_{B_r; k=1,\cdots,n}  r|\partial_k g_{ij}| +\sup_{x\neq y; k=1,\cdots,n} \left(r^{1+\sigma}\frac{|\partial_k g_{ij}(x)-\partial_k g_{ij}(y)|}{|x-y|^\sigma}\right)$$ 
and $r:=\epsilon I$. (The notation $|\cdot|^{'}$ is adopted from \cite[Page 53]{GilbargTrudinger}.)
On the other hand, one can compute the Laplacian of $g_{ij}$ under harmonic coordinates and derive 
$$\Delta_g g_{ij} = -2 R_{ij} + P(g_{ij},\partial g_{ij}),$$
where $P$ is a certain quasi-polynomial of $g_{ij}$ and $\partial g_{ij}$ (cf. \cite[Chapter 10]{Petersen}). 
Hence the standard elliptic regularity theory  (cf. \cite[Theorem 4.6]{GilbargTrudinger}) tells that 
\begin{equation}
|g_{ij}|^{'}_{C^{2,\sigma}} \leq C \left(  |R_{ij}|^{'}_{C^{0,\sigma}}+|g_{ij}|^{'}_{C^{1,\sigma}} \right).
\end{equation}
Note that when the Ricci curvature and its derivative are bounded in the sense that $|Ric|^2=g^{ik}g^{jl}R_{ij}R_{kl}\leq K^2$ and $|\nabla Ric|^2= g^{pq}g^{ik}g^{jl}\nabla_pR_{ij}\nabla_qR_{kl}\leq L$,
then the norm of coefficients $R_{ij}$ and its derivatives shall satisfy $|R_{ij}|^{'}_{C^{0,\sigma}}\leq C=C(K,\epsilon I,L)$. Therefore, one can use harmonic coordinates and (1) to derive a bound for $|g|^{'}_{C^{2,\sigma}}$.
In particular, the coefficients of curvature tensor are bounded. Since the tensor norm does not depend on coordinate choosing, the curvature $R_{ij\ l}^{\ \ k}$ is bound in the tensor sense. It is also equivalent to say that curvature operator $Rm$ is bounded.
Such strategy will be used several times in this article. 
One should be cautious that, for the elliptic regularity on Riemannian manifolds,
the constant $C$ depends not only on $n$, $\sigma$, but also on the upper bound of $|g^{ij}|$. 
(One can see this when adapting proofs of theorems in Chapter 2, 3 and 4 of \cite{GilbargTrudinger} into the Riemannian case.) 
Hence $C=C(n,\sigma,\epsilon,I)$.

For the reader's convenience, we prove the following compact version of Theorem \ref{MTL} first. 
The proof of the standard version of Theorem \ref{MTL} is more subtle and will be demonstrated in the next section.

\begin{Thm*MTL}[Compact version] 
For any $\delta,\eta>0$ and $n\in\mathbb{N}$, there exist positive constants $\alpha$ and $C$ such that 
for any $K>0$ and any closed smooth Ricci flow  $(M^n,g(t))_{t\in[0,T]}$ with $T\geq\frac{\eta}{K}>0$,
if 
$$|Ric|\leq K \ \ \mbox{ and }\ \ inj \geq \delta K^{-\frac{1}{2}}\ \ \mbox{ for all }t\in(0,T],$$
then 
$$|\nabla Ric| \leq \alpha  \left( KTt^{-1} \right)^{\frac{3}{2}}\ \mbox{ and }\ \ |Rm|\leq CKTt^{-1}\ \mbox{ for all } t\in(0,T].$$
\end{Thm*MTL}

\begin{Lma} \label{MLG}
For any $\delta,\eta>0$ and $n\in\mathbb{N}$, there exists $\alpha>0$ such that 
for any smooth Ricci flow $g(t)_{t\in[0,T]}$ on a closed manifold $M^n$, if $|Ric|\leq K$, $inj\geq \frac{\delta}{\sqrt{K}}$ and $T\geq \frac{\eta}{K}$, then 
$$|\nabla Ric| \leq \alpha \left(KTt^{-1}\right)^{\frac{3}{2}}$$
for all $t\in\left(0,T\right]$.
\end{Lma}

\begin{proof}
Suppose no such $\alpha$ exists, then we can find a sequence of Ricci flows $g_k(t)_{t\in[0,T_k]}$, points $p_k=(x_k,t_k)$, and $\alpha_k\nearrow\infty$ such that $t_k>0$ and
$|\nabla Ric|_{g_k}(p_k) > \alpha_k \left( K_kT_kt_k^{-1} \right)^{\frac{3}{2}}$. 
By the point-picking lemma afterwards, we can find $\bar p_k=(\bar x_k,\bar t_k)$ associated to $p_k$ such that
\begin{itemize}
\item $|\nabla Ric|_{g_k}(\bar p_k) > \alpha_k \left( K_kT_k\bar t_k^{-1} \right)^{\frac{3}{2}}$ and
\item $|\nabla Ric|_{g_k} \leq 8 \bar Q_k:= 8|\nabla Ric|_{g_k}(\bar p_k)$ on $M\times [\bar t_k-\beta_k\bar Q_k^{-\frac{2}{3}} ,\bar t_k]$, where $\beta_k:=\frac{1}{2}\alpha_k^{\frac{2}{3}}\eta$.
\end{itemize}
Consider the rescaling Ricci flows $\widetilde g_k := \bar Q_k^{\frac{2}{3}} g_k$ with $\widetilde t := \bar Q_k^{\frac{2}{3}} (t -\bar t_k) \in[-\beta_k,0]$.
Then $|Ric_{\widetilde{g}_k}|_{\tilde{g}_k}\leq K_k \bar Q_k^{-\frac{2}{3}}\leq \alpha_k^{-\frac{2}{3}}\frac{t_k}{T_k} \searrow 0$ and $|\nabla Ric_{\widetilde{g}_k}|_{\widetilde{g}_k}\leq 8$ on $M\times [-\beta_k,0]$.
In particular, $Ric $ has a uniform $C^{0,\sigma}$-bound.

Using Anderson's lemma mentioned before, 
$inj_{\widetilde{g}_k}\geq  \frac{\delta}{\sqrt{K_k}}Q_k^{\frac{1}{3}} =   \frac{\delta}{\sqrt{K_k}}\alpha_k^{\frac{1}{2}}\left( K_kT_k\bar t_k^{-1} \right)^{\frac{1}{2}} \nearrow\infty$ and the boundedness of $|Ric_{\widetilde{g}_k}|_{\widetilde{g}_k}$ ensures the existence of harmonic coordinates on a domain of uniform size. Moreover, $\widetilde{g}_k$'s may have a uniform $C^{1,\sigma}$-bound in this domain.
By elliptic regularity, $\widetilde{g}_k$'s, which satisfy $\Delta\widetilde{g}_k=-2Ric_{\widetilde{g}_k}+P(\widetilde{g}_k,\partial\widetilde{g}_k)$, have a uniform $C^{2,\sigma}$-bound. Namely, $|Rm_{\widetilde{g}_k}|_{\widetilde{g}_k}$'s are uniformly bounded on $M\times (-\beta_k,0]$. 
Applying Shi's estimate, all higher derivatives of $Rm_{\widetilde{g}_k}$'s are uniformly bounded on $M\times [-\frac{1}{2}\beta_k,0]$. So the marked metrics $(\widetilde g_k, \bar p_k)$ converge smoothly to a Ricci flat metric $(g_\infty, p_\infty)$. This contradicts $|\nabla Ric_{g_\infty} |_{g_\infty}(p_\infty)=1$.
Therefore, we have $|\nabla Ric| \leq \alpha  \left( KTt^{-1} \right)^{\frac{3}{2}}$.
\end{proof}

\begin{Lma}[Point-picking lemma]
For any $\alpha>0$ and any closed smooth Ricci flow $(M,g(t))_{t\in\left[0,T\right]}$ with $|Ric|\leq K$ and $T\geq\frac{\eta}{K}$,
if $|\nabla Ric| (p_0)> \alpha  \left( KTt_0^{-1} \right)^{\frac{3}{2}}$ at some point $p_0=(x_0,t_0)$ with $t_0>0$, 
then there exists $\bar p=(\bar x,\bar t), \bar t>0$, such that 
$$|\nabla Ric|(\bar p) > \alpha \left( KT\bar t^{-1} \right)^{\frac{3}{2}}\ \ \mbox{ and }\ \ \
|\nabla Ric| \leq 8 \bar Q:= 8 |\nabla Ric|(\bar p)$$ on $M\times [\bar t-\beta\bar Q^{-\frac{2}{3}} ,\bar t]$, where $\beta:=\frac{1}{2}\alpha^{\frac{2}{3}}\eta$.
\end{Lma}

\begin{proof}

Here we use Perelman's method for proving his pseudo-locality theorem (cf. \cite[Theorem 10.1]{Perelman02}).
Start from the point $p_0$ with  $Q_0:=|\nabla Ric| (p_0)> \alpha \left( KTt_0^{-1} \right)^{\frac{3}{2}}$. 
If $|\nabla Ric| \leq 8Q_0$ on $M\times [ t_0-\beta Q_0^{-\frac{2}{3}} , t_0]$, then we are done.
Suppose this is not the case, then there exists a point $p_1=(x_1,t_1)$ with $t_1\in[t_0-\beta Q_0^{-\frac{2}{3}}, t_0] $ 
and $Q_1:=|\nabla Ric| (p_1) > 8Q_0$. 
\begin{figure}[h]
\includegraphics[bb= 480 40 0 220 ]{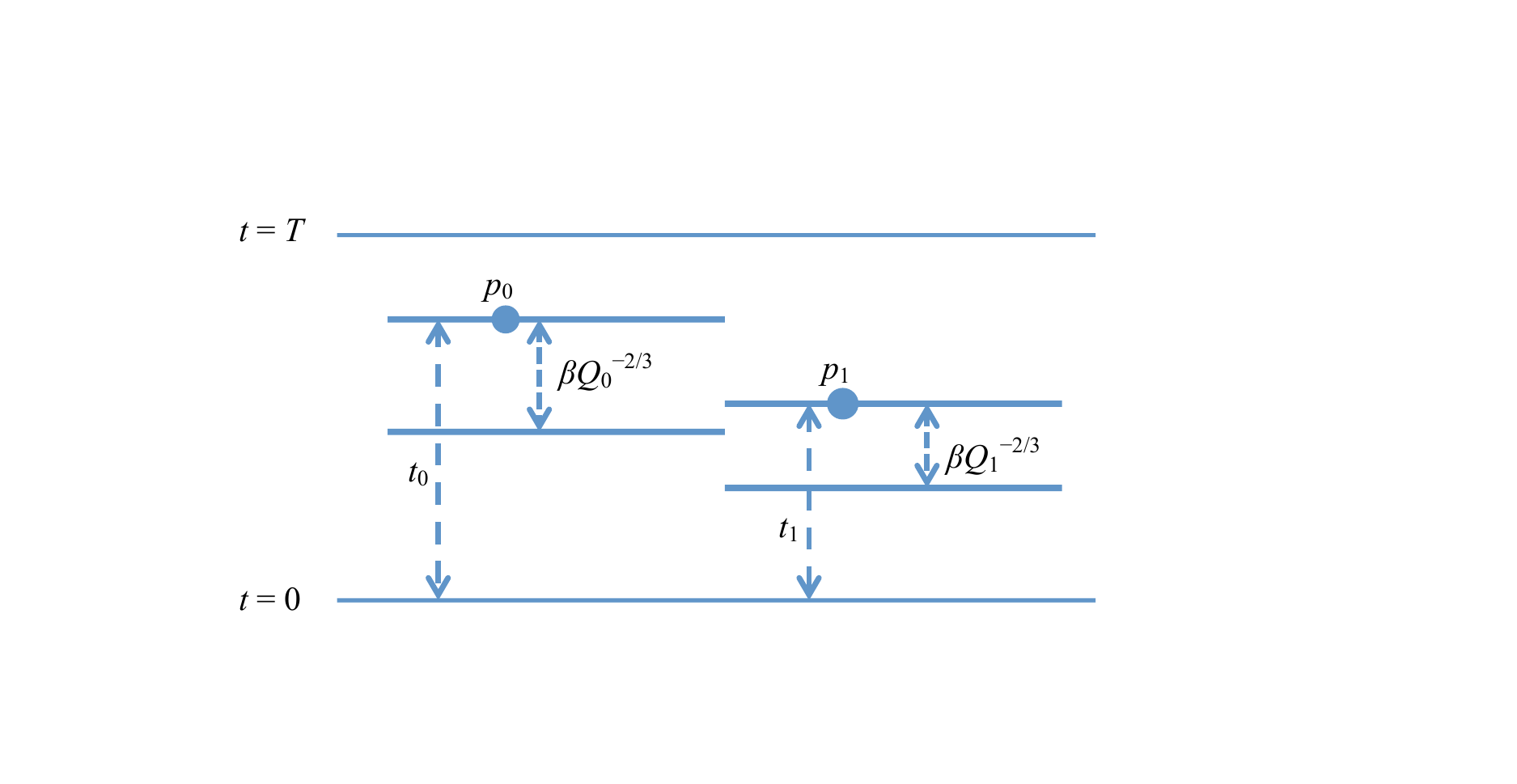}
\end{figure}
Note that $\beta=\frac{1}{2}\alpha^{\frac{2}{3}}\eta $ implies that $\beta Q_0^{-\frac{2}{3}}\leq \frac{1}{2}t_0$. 
Thus, $t_1 \geq t_0   -\beta Q^{-\frac{2}{3}} \geq \frac{1}{2} t_0$.
In particular, $$Q_1>8Q_0>8 \alpha  \left( KTt_0^{-1} \right)^{\frac{3}{2}} \geq \alpha  \left( KTt_1^{-1} \right)^{\frac{3}{2}}$$ 
and thus 
$$ \beta Q_1^{-\frac{2}{3}}  \leq \frac{1}{2} t_1.$$
If $|\nabla Ric| \leq 8Q_1$ on $M\times [ t_1-\beta Q_1^{-\frac{2}{3}} , t_1]$, then we are done.
Suppose not, then we can further find $p_2$ so that $Q_2>8Q_1> \alpha  \left( KTt_2^{-1} \right)^{\frac{3}{2}}$ by similar process.
Similarly,  $\beta Q_2^{-\frac{2}{3}} \leq \frac{1}{2}t_2$  and so on. 
So $t_k$ always stays in $(0,T]$.
Therefore, such process could be continued until we find a $p_k$ so that $|\nabla Ric|\leq 8Q_k$ in $M\times [ t_k-\beta Q_k^{-\frac{2}{3}} , t_k]$.
Such $p_k$ must exist because $|\nabla Ric|(p_k)>8^kQ_0$ must be bounded in $M\times [0,T]$.
\end{proof}

Now we are able to finish the proof of the compact version of Theorem \ref{MTL}.

\begin{proof}
For a Ricci flow $(M^n,g(t))$ with $|Ric|\leq K$ and $inj\geq \delta K^{-\frac{1}{2}}>0$ for all $t\in[0,T]$, by Lemma \ref{MLG}, we have $|\nabla Ric| \leq \alpha (KTt^{-1})^{\frac{3}{2}}$ for all $t\in\left(0,T\right]$.
For each fixed $t\in\left(0,T\right]$, we consider $\widetilde g := (KTt^{-1})g(t)$ 
and obtain 
$$inj_{\widetilde{g}}\geq \delta(Tt^{-1})^{\frac{1}{2}}\geq \delta, \ \ |Ric_{\widetilde{g}}|_{\widetilde{g}}\leq T^{-1}t\leq 1\ \mbox{ and }\ |\nabla Ric_{\widetilde{g}}|_{\widetilde{g}} \leq \alpha. $$
So by the elliptic regularity (1), the metric tensor $\widetilde{g}$ has a uniform $C^{2,\sigma}$-bound which depends on $n,\sigma, \delta$ and $\alpha$. By choosing an arbitrary $\sigma\in(0,1)$, $|Rm_{\widetilde{g}}|_{\widetilde{g}}$ is bounded by a constant depending only on $n, \delta$ and $\alpha$. After rescaling back, we see that the curvature of $g(t)$ satisfies $|Rm|\leq CKTt^{-1}$ with $C=C(n, \delta, \alpha)$. Since $t$ is arbitrary in $(0,T]$ and $\alpha=\alpha(\delta,\eta,n)$, the theorem is proved. 
\end{proof}

\section{Local estimate of $Rm$}

The estimate in the previous section also holds locally. Namely, the curvature operator can be bounded if $|Ric|$ and $inj$ are bounded in a parabolic neighborhood of uniform size. To prove this, we need the following local point-picking lemma which shows that for any point with large $|\nabla Ric|$, one can find another point nearby equipped with a controlled parabolic neighborhood.

\begin{Lma}[point-picking lemma, local version]
For any $\eta, \alpha>0$ and any Ricci flow $(B_{4\sqrt{T}}(x_0,t),g(t))_{t\in\left[0,T\right]}$
with $|Ric|\leq K$ and $T\geq\frac{\eta}{K}$, which is smooth up to boundary, 
if $|\nabla Ric| (p)> \alpha  \left( KTt^{-1} \right)^{\frac{3}{2}}$ at some point $p=(x,t)$ 
in $(B_{2\sqrt{T}}(x_0,t),g(t))_{t\in\left(0,T\right]}$, then 
there exist $\epsilon=\epsilon(\eta)>0$ and $\bar p=(\bar x,\bar t)$ with $d_{\bar t}(\bar x, x_0)<4\sqrt{T}$ and $\bar t>0$, such that 
$$|\nabla Ric|(\bar p) > \alpha \left( KT\bar t^{-1} \right)^{\frac{3}{2}}\ \ \mbox{ and }\ \ \
|\nabla Ric| \leq 8 \bar Q:= 8 |\nabla Ric|(\bar p)$$ in 
$B_{\beta^{\frac{1}{2}} \bar Q^{-\frac{1}{3}}}(\bar x,t), t\in[\bar t-\beta\bar Q^{-\frac{2}{3}} ,\bar t]$, where $\beta:=\frac{1}{2}\epsilon^2\alpha^{\frac{2}{3}}\eta$. 
In particular, $B_{\beta^{\frac{1}{2}} \bar Q^{-\frac{1}{3}}}(\bar x,t)\subset B_{4\sqrt{T}}(x_0,t)$ for each $t\in[\bar t-\beta\bar Q^{-\frac{2}{3}} ,\bar t]$. 
\end{Lma}

\begin{proof}
To abbreviate the notation, we define the backward parabolic metric ball based at $(x_*,t_*)$ by
$$\mathcal{B}(r;x_*,t_*):= \bigcup_{t\in(t_*-r^2, t_*]} B_{r}(x_*,t)= \{ (x,t) | dist_{g(t)}(x,x_*)< r, t\in(t_*-r^2 ,t_*] \}.$$
We use the same induction argument as in the proof of the global point-picking lemma. 
In the proof of the global version, we have seen that to proceed the argument, 
we need to justify that $p_k$'s can stay in a finite time-region.
Here the situation is more subtle: we need to make sure that there exists an $\epsilon>0$ such that
$\mathcal{B}(\beta^{\frac{1}{2}} Q_k^{-\frac{1}{3}};x_k,t_k) $ is contained in $ \bigcup_{t\in(0,T]} B_{4\sqrt{T}}(x_0,t)$,
where $\beta:=\frac{1}{2}\epsilon^2\alpha^{\frac{2}{3}}\eta$.

Now we start from $p=(x,t)$ and look at the parabolic region $\mathcal{B}(\beta^{\frac{1}{2}} Q^{-\frac{1}{3}};x,t)$.
If $\mathcal{B}(\beta^{\frac{1}{2}} Q^{-\frac{1}{3}};x,t)\subset \bigcup_{t\in(0,T]} B_{4\sqrt{T}}(x_0,t)$ and $|\nabla Ric| \leq 8 Q:= 8 |\nabla Ric|( p)$ in this region, then we are done. 
If not, then there is a $p_1\in\mathcal{B}(\beta^{\frac{1}{2}} Q^{-\frac{1}{3}};x,t)$ such that $Q_1:=|\nabla Ric|(p_1) > 8Q$. Similar to the proof of the compact version, we wish to go on finding successive $p_k$'s until we acquire $\bar p$. So we should check 
\begin{itemize}
\item $t_k$'s will not reach $0$:
$\beta \cdot Q_k^{-\frac{2}{3}} \leq \frac{1}{2} \epsilon^2\alpha^{\frac{3}{2}} \eta  \cdot \alpha^{-\frac{2}{3}} K_k^{-1}T_k^{-1} t_k \leq\frac{1}{2}\epsilon^2t_k$
\item $x_k$'s stay in a distance less than $3\sqrt{T}$ from the center $x_0$ at each time $t_k$:
\begin{align*}
d_{t_k}(x_k,x_0) & \leq d_{t_{k-1}}(x_{k-1},x_0)+\sqrt{\beta Q_{k-1}^{-\frac{2}{3}}}\\
                           & \leq d_{t_{k-2}}(x_{k-2},x_0)+\sqrt{\beta Q_{k-2}^{-\frac{2}{3}}}+\sqrt{\beta Q_{k-1}^{-\frac{2}{3}}}\\
                           & \leq  d_{t_1}(x_1,x_0)+\sqrt{\beta Q_1^{-\frac{2}{3}}}   +\cdots +\sqrt{\beta Q_{k-1}^{-\frac{2}{3}}}\\
                           & \leq  2\sqrt{T}+\sqrt{\beta Q^{-\frac{2}{3}}}+\sqrt{\beta Q_1^{-\frac{2}{3}}}   +\cdots +\sqrt{\beta Q_{k-1}^{-\frac{2}{3}}}\\
                           & <2\sqrt{T}+\sqrt{\beta Q^{-\frac{2}{3}}}+ \sqrt{\frac{1}{4}\beta Q^{-\frac{2}{3}}}   +\cdots +\sqrt{ \frac{1}{4^{k-1}}\beta Q^{-\frac{2}{3}}}\\
                           & < 2\sqrt{T}+2 \sqrt{\beta Q^{-\frac{2}{3}}}\\
                           & <2\sqrt{T}+ 2 \epsilon \sqrt{\frac{1}{2}t}\\
                           & <3 \sqrt{T}, \mbox{ if } \epsilon \mbox{ is small enough, for instance, less than }\frac{1}{\sqrt{2}}. 
\end{align*}

\begin{figure}[h]
\includegraphics[bb= 480 60 0 220 ]{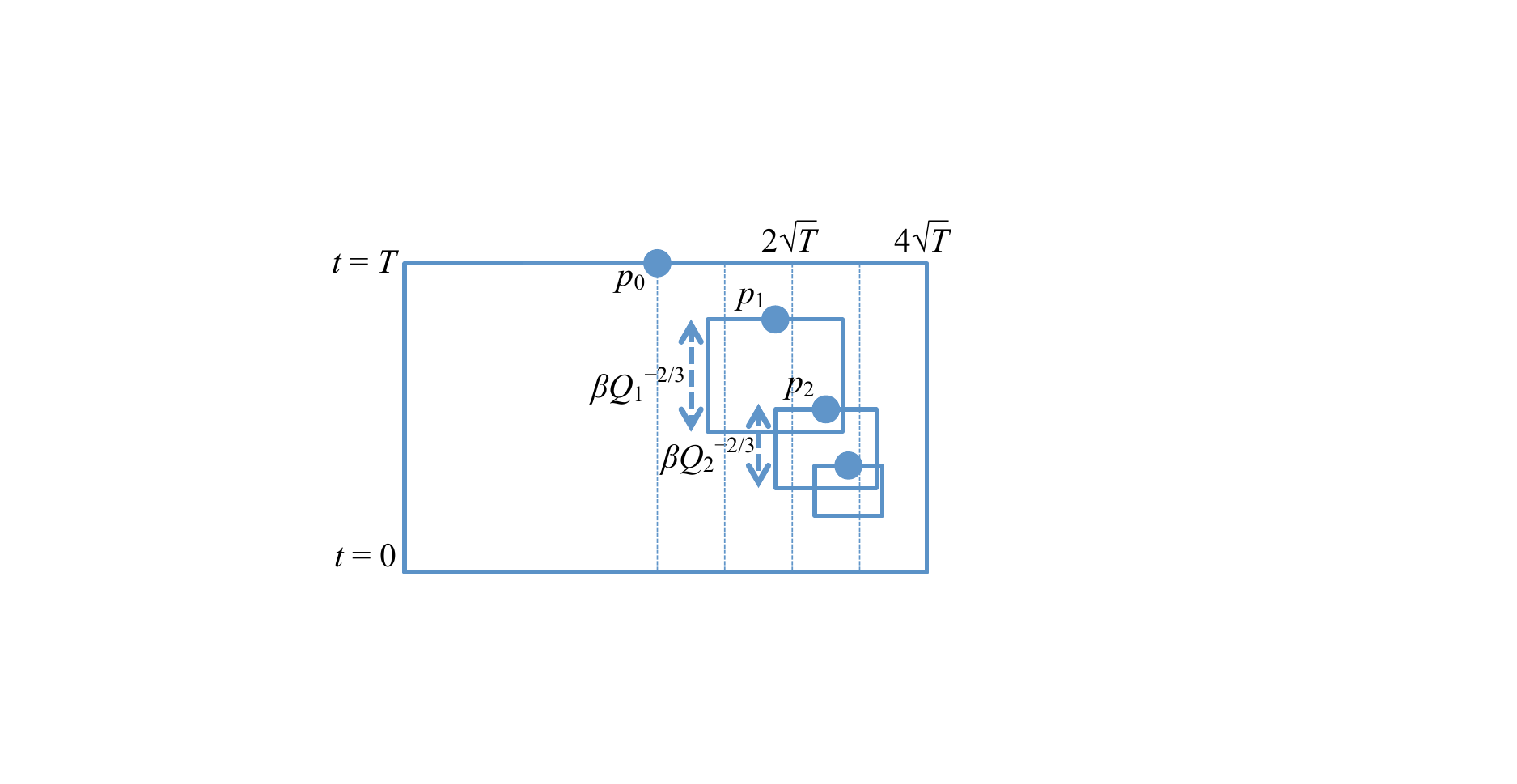}
\end{figure}

\item $\mathcal{B}(\beta^{\frac{1}{2}} Q_k^{-\frac{1}{3}};x_k,t_k)\subset  \bigcup_{t\in(0,T]} B_{4\sqrt{T}}(x_0,t)$:

For any $q=(\xi,\tau)\in\mathcal{B}(\beta^{\frac{1}{2}} Q_k^{-\frac{1}{3}};x_k,t_k)$, since 
$$K|t_k-\tau|\leq K \beta \cdot Q_k^{-\frac{2}{3}}\leq \frac{1}{2}\epsilon^2 \eta T^{-1}t_k \leq  \frac{1}{2}\epsilon^2 \eta ,$$ 
we have
\begin{align*}
d_{\tau}(\xi, x_0) & \leq  d_{\tau}(\xi, x_k)+d_{\tau}(x_k, x_0) \\
                          & \leq   \beta^{\frac{1}{2}} Q_k^{-\frac{1}{3}} + d_{t_k}(x_k,x_0)\cdot e^{K|t_k-\tau|} \\
                          & \leq  \epsilon\sqrt{\frac{1}{2} t_k} +3\sqrt{T}\cdot e^{\frac{1}{2}\epsilon^2 \eta } \\
                          &\leq \sqrt{T}\left(\sqrt{\frac{1}{2}} \epsilon+3\cdot e^{\frac{1}{2}\epsilon^2 \eta }\right)\\
                          &\leq 4\sqrt{T}, \mbox{ if }\epsilon = \epsilon(\eta) \mbox{ is small enough.}
\end{align*}
\end{itemize}

Therefore, this local point-picking lemma follows by the same argument of induction as in the proof of Lemma 2.
\end{proof}

Now we demonstrate the proof of Theorem \ref{MTL} and Theorem  \ref{MTLh}.

\begin{proof}[Proof of Theorem \ref{MTL}]

The idea of proof is the same to the compact version, so we will be a bit sketchy on the whole process but focus on 
the crucial steps. 
Suppose there exist $(B_{4\sqrt{T}}(x_0,t),g_k(t))_{t\in[0,T]}$, with points $x_k\in B_{2\sqrt{T}}(x_0,t_k)$ and $\alpha_k\to\infty$ such that 
the first conclusion is not true, i.e., $Q_k:=|\nabla Ric|_{g_k}(p_k)>\alpha_k(K_kT_k t_k^{-1})^{\frac{3}{2}}$.
By the local point-picking lemma, we can find $\bar{p}_k$'s and the associated parabolic regions  
$\mathcal{B}( \beta^{\frac{1}{2}}\bar{Q}_k^{\frac{1}{3}};\bar{x}_k, \bar{t}_k)$'s to run the blow-up procedure,
where $\beta:=\frac{1}{2}\epsilon^2\alpha^{\frac{2}{3}}\eta$. 
Indeed, we rescale the metric $g_k(t)$ on $\mathcal{B}( \beta^{\frac{1}{2}}\bar{Q}_k^{\frac{1}{3}};\bar{x}_k, \bar{t}_k)$
by $\bar{Q}_k^{\frac{2}{3}}:=(|\nabla Ric|_{g_k}(\bar{p}_k))^{\frac{2}{3}}>\alpha_k^{\frac{2}{3}}K_kT_kt_k^{-1}$ and obtain 
$$|Ric_{\widetilde{g}_k}|_{\tilde{g}_k}\leq K_k \bar Q_k^{-\frac{2}{3}}\leq \alpha_k^{-\frac{2}{3}}\frac{\bar{t}_k}{T_k} \searrow 0$$ 
and 
$$inj_{\widetilde{g}_k} \geq \delta \cdot \alpha_k^{\frac{1}{3}}(T_kt_k^{-1})^{\frac{1}{2}} \geq \delta \alpha_k^{\frac{1}{3}}\nearrow \infty$$
on $\widetilde{\mathcal{B}}( \beta^{\frac{1}{2}};\bar{x}_k, \bar{t}_k)$, as $\alpha_k \nearrow \infty$.

As in the proof of the compact version, we encounter a contradiction on a sub-sequential limit. 
Therefore, there exists $\alpha>0$ such that $|\nabla Ric| \leq \alpha(KTt^{-1})^{\frac{3}{2}}$ for all Ricci flows $(B_{2\sqrt{T}}(x_0,t),g(t))_{t\in[0,T]}$.

Fix an arbitrary $t$ and rescale the metric by letting $\tilde g := KTt^{-1} g(t)$, as in the proof of the compact version, 
one obtains
$$inj_{\widetilde{g}}\geq \delta(Tt^{-1})^{\frac{1}{2}}\geq \delta, \ \ |Ric_{\widetilde{g}}|_{\widetilde{g}}\leq T^{-1}t\leq 1\ \mbox{ and }\ |\nabla Ric_{\widetilde{g}}|_{\widetilde{g}} \leq \alpha$$
in $\widetilde{B}_{2\sqrt{\eta}}(x_0,t)$.
So by the elliptic regularity (1), $|Rm_{\widetilde{g}}|_{\widetilde{g}}\leq C$ in $\widetilde{B}_{\rho}(x_0,t)$
where $C=C(n, \delta,\alpha)$ and $\rho=\rho(n,\delta,\eta)$. After rescaling back, we see that the curvature of $g(t)$ satisfies $|Rm|\leq CKTt^{-1}$ in $B_{\rho\sqrt{K^{-1}T^{-1}t}}(x_0,t)$ with $C=C(n, \delta, \alpha)$ and thus prove the theorem. 
\end{proof}

\begin{proof}[Proof of Theorem \ref{MTLh}]
The key of this proof is the following observation:
in the proof of local point-picking lemma, one can see that the parabolic neighborhood associated to $\bar p$ must be small if 
$\bar t$ is close to $0$. Precisely, the radius is $\beta^{\frac{1}{2}}\bar Q^{-\frac{1}{3}}$ and its square is less than $\frac{1}{2}\bar t$.
Hence, when performing the blow-up argument along these picked points, we do not need a uniform lower bound of injectivity radius. 
Instead, the injectivity radius is allowed to decay at a rate proportional to $\bar Q^{-\frac{1}{3}}:=  \alpha_k^{-\frac{1}{3}}(K_kT_kh_k^{-1}(\bar t_k))^{-\frac{1}{2}}$, i.e., either $inj \geq K_k^{-\frac{1}{2}}$ or  $inj(x,t)\geq h_k^{\frac{1}{2}}(t)$ is enough.
It is not hard to check that, if the condition $|\nabla Ric|(p)\leq (\alpha KTt^{-1})^{\frac{3}{2}}$ in the local point-picking lemma 
is replaced by $|\nabla Ric|(p)\leq (\alpha KTh^{-1}(t))^{\frac{3}{2}}$, for some positive function $h(t)\leq t$ for all $t\in(0,T]$,
then the lemma still holds with the conclusion replaced by $|\nabla Ric|(\bar p)\leq (\alpha KTh^{-1}(\bar t))^{\frac{3}{2}}$. 
Now we use this modified version to prove Theorem \ref{MTLh}.

Again we argue by contradiction. Suppose there exist $(B_{4\sqrt{T}}(x_0,t),g_k(t))_{t\in[0,T]}$, with points $x_k\in B_{2\sqrt{T}}(x_0,t_k)$ and $\alpha_k\to\infty$ such that 
$Q_k:=|\nabla Ric|_{g_k}(p_k)>\alpha_k(K_kT_k h_k^{-1})^{\frac{3}{2}}$.
By the local point-picking lemma (with $t$ being replaced by $h$ in the bound of $\nabla Ric$), we can find $\bar{p}_k$'s and the associated parabolic regions  
$\mathcal{B}( \beta^{\frac{1}{2}}\bar{Q}_k^{\frac{1}{3}};\bar{x}_k, \bar{t}_k)$'s to run the blow-up procedure,
where $\beta:=\frac{1}{2}\epsilon^2\alpha^{\frac{2}{3}}\eta$. 
Indeed, we rescale the metric $g_k(t)$ on $\mathcal{B}( \beta^{\frac{1}{2}}\bar{Q}_k^{\frac{1}{3}};\bar{x}_k, \bar{t}_k)$
by $\bar{Q}_k^{\frac{2}{3}}:=(|\nabla Ric|_{g_k}(\bar{p}_k))^{\frac{2}{3}}>\alpha_k^{\frac{2}{3}}K_kT_kh_k^{-1}(\bar{t}_k)$ and obtain 
$$|Ric_{\widetilde{g}_k}|_{\tilde{g}_k}\leq K_k \bar Q_k^{-\frac{2}{3}}  \leq \alpha_k^{-\frac{2}{3}}\frac{\bar{h}_k}{T_k}  \leq \alpha_k^{-\frac{2}{3}}\frac{\bar{t}_k}{T_k}\searrow 0$$ 
and 
$$inj_{\widetilde{g}_k} \geq \delta \cdot\min\{ K_k^{-\frac{1}{2}}, h_k^{\frac{1}{2}}(t)\} \cdot \alpha_k^{\frac{1}{3}}(K_kT_kh_k^{-1}(\bar{t}_k))^{\frac{1}{2}}.$$ 
To derive a uniform lower bound from the second inequality, we study the following two cases:
\begin{itemize}
\item At points where $inj_{g_k} \geq \delta K_k^{-\frac{1}{2}}$, we have 
$$inj_{\widetilde{g}_k} \geq \delta \cdot   K_k^{-\frac{1}{2}}\cdot \alpha_k^{\frac{1}{3}}(K_kT_kh_k^{-1}(\bar{t}_k))^{\frac{1}{2}}\geq \delta \cdot   \alpha_k^{\frac{1}{3}}(T_kh_k^{-1}(\bar{t}_k))^{\frac{1}{2}}\nearrow \infty.$$ 
\item At points where $inj_{g_k(t)} \geq \delta h_k^{\frac{1}{2}}(t)$, we have  
$$inj_{\widetilde{g}_k} \geq \delta h_k^{\frac{1}{2}}(t)\cdot \alpha_k^{\frac{1}{3}}(K_kT_kh_k^{-1}(\bar{t}_k))^{\frac{1}{2}} \geq m^{-\frac{1}{2}}\delta \eta^{\frac{1}{2}} \alpha_k^{\frac{1}{3}} \nearrow \infty$$ by the assumption of $h_k(t)$.
\end{itemize}
In either case, the rescaled injectivity radius has a uniform lower bound on  $\widetilde{\mathcal{B}}( \beta^{\frac{1}{2}};\bar{x}_k, \bar{t}_k)$.
So a contradiction can be derived as before and one can conclude that 
$|\nabla Ric|\leq \alpha(KTh^{-1})^{\frac{3}{2}}$.
At last, by using a rescaling argument as in the proof of Theorem \ref{MTL}, one achieves  $|Rm|\leq CKTh^{-1}$ for some $C=C(n,\delta,\eta,m)$ in $B_{\rho\sqrt{K^{-1}T^{-1}h(t)}}(x_0,t)$ for every fixed $t$ and thus the theorem is proved.
\end{proof}

\begin{Rmk}
The proto type for the function $h$ is $h(t)=t$. In general, if $h(t)$ is a concave function or a decreasing function, or a conjunction of them, then it satisfies the assumption that there exists $m>0$ such that for all $t^*\in (0,T]$, $h(t)\geq m^{-1}\cdot h(t^*)$ for all $t\in\left[\frac{1}{2}t^*,t^*\right]$. For instance, $h(t)$ could be $\frac{T}{\pi}\sin \frac{\pi t}{T}$. 
\end{Rmk}

\section{A geometrical alternative to De Giorgi-Nash-Moser's iteration}

In this section, we compare our geometrical blow-up argument with Moser's iteration technique.
In particular, by using Moser's iteration and results in \cite{KotschwarMunteanuWang15},
we derive a theorem which requires a weaker injectivity radius assumption, and has
weaker conclusion, than Theorem \ref{MTL}. 
We first recall the following crucial lemma \cite[Proposition 1]{KotschwarMunteanuWang15}. 
\begin{Prop} \label{KMWProp1}
Let $(M^n,g(t))$ be a smooth solution to the Ricci flow defined for $0\leq t\leq T$. 
Assume that there exist $A$, $K>0$ such that 
$|Ric|\leq K  \mbox{ on } \overline{\mathcal{P}(\frac{A}{\sqrt{K}};x_0,T)}:=B_{\frac{A}{\sqrt{K}}} \left( x_0,0\right)\times [0,T]$.
Then, for any $p\geq 3$, there exists $c=c(n,p)>0$ so that for all $0\leq t\leq T$
$$\|Rm\|_{L^p(B')}^p (t)\leq ce^{cKT} \left( \|Rm\|_{L^p(B)}^p (0)
+ K^p(1+A^{-2p}) Vol_{g(t)}(B) \right),$$ 
where $B' = B_{\frac{A}{2\sqrt{K}}} (x_0,0)$ and $B=B_{\frac{A}{\sqrt{K}}} (x_0,0)$. 
\end{Prop} 
By using this proposition, Moser's technique and its generalized version for varying metrics (cf. \cite[Ch. 19]{Li12}, \cite{Yang8788, Ye08} or \cite[Theorem 2.1]{DaiWeiYe96}), Kotschwar-Munteanu-Wang derived the following bound:
$$ |Rm|(x_0,T)\leq ce^{c(KT+A)} \left(  1+ \left(\frac{\Lambda_0}{K}\right)^\alpha +\left( \frac{1}{KT} +A^{-2} \right)^\beta  \right)(K(1+A^{-2}) +\Lambda_0 ).$$
Here $\alpha,\beta,c$ are constants depending only on $n$ and $ \Lambda_0= \sup_{B_{\frac{A}{\sqrt{K}}} (x_0,0)} |Rm|$.

To compare their approach with ours, we prove the following theorem based on Kotschwar-Munteanu-Wang's argument. 
\begin{Thm*KMW} 
Let $(M^n,g(t))_{t\in[0,T]}$ be a smooth solution of the Ricci flow. 
If $$|Ric|\leq K  \mbox{ in } \overline{\mathcal{P}(4r;x_0,T)} \ \mbox{ and }\ 
\inf_{B_{4r} \left( x_0,0 \right)} inj\geq I>0$$
for some $K,I,r > 0$,
then there exists $C$ depending on $K,I,r,T$ and the dimension $n$ such that 
$|Rm|(x_0,t)\leq C$ in $\mathcal{P}(r;x_0,T)\setminus \mathcal{P}(r;x_0,\frac{1}{2}T)$.
\end{Thm*KMW}

\begin{proof}
For any point $y\in B_{2r}(x_0,0)$, consider the harmonic coordinates around it.
Since Ricci curvature is bounded and the injectivity radius is bounded from below,
these coordinates cover a geodesic ball $B_{\epsilon I}(y,0)$ for some $\epsilon$ depending on $K,I$ and $n$.
Since $\|Ric\|_{L^p}$ is bounded, elliptic regularity of the equation $\Delta_{g} g_{ij}=-2 R_{ij} +P(g,\partial g)$ shows that 
$\|Rm\|_{L^p(B_{\epsilon I}(y,0))}$ is bounded by some constant $C$ for all $p >1$ (cf. Section 2 or \cite{Anderson90}). 
Furthermore, a standard result due to Gromov says that the lower bound of Ricci curvature implies that $B_{2r}(x_0,0)$ 
can be covered by a finite collection of $B_{\epsilon I}(y_i,0)$'s, say $i=1,\dots,N$. 
Note that $N$ only depends on $n,\epsilon I,r$ and $K$.
Hence $\|Rm\|_{L^p(B_{2r}(x_0,0))}$ must be bounded by some constant $C_L=C_L(n,K,I,r,p)$.

By Proposition \ref{KMWProp1}, when $p\geq 3$, the evolving $L^p$-norm of $|Rm|$ on $B_{r}(x_0,0)$ is controlled by
$$\|Rm\|_{L^p(B_r)}^p (t)\leq ce^{cKT} \left( \|Rm\|_{L^p(B_{2r})}^p (0)
+ K^p(1+(\epsilon I\sqrt{K})^{-2p}) Vol_{g(t)}(B_{2r}) \right),$$ 
where $B_r= B_{r} (x_0,0)$ and $B=B_{2r} (x_0,0)$. 
Divide both sides by $Vol_{g(t)}(B_r)$ and use $\|Rm\|_{L^p(B_{2r})}^p (0)\leq C_L^p$, one obtains
$$\frac{1}{Vol_{g(t)}(B_r)} \|Rm\|_{L^p(B_r)}^p (t)
\leq \frac{ce^{cKT}C_L^p }{Vol_{g(t)}(B_r)}+ ce^{cKT} (K^p+(\epsilon I)^{-2p}) \cdot\frac{Vol_{g(t)}(B_{2r}) }{Vol_{g(t)}(B_r)}.$$
To find a lower bound for $Vol_{g(t)}(B_r)$, we need Berger-Croke's theorem (cf. \cite[Proposition 14]{Croke80}):
there is a uniform constant $C_{inj}$ which depends only on $n$ such that, for any Riemannian metric and any $r>0$, 
$\displaystyle \inf_{B_{2r}} inj\geq I$ implies that $Vol(B_r)\geq C_{inj} I^n$. Thus
\begin{align*}
\frac{1}{Vol_{g(t)}(B_r)} \|Rm\|_{L^p(B')}^p (t)
& \leq ce^{cKT}C_L^p C_{inj}I^n+ ce^{cKT} (K^p+(\epsilon I)^{-2p}) \cdot\frac{Vol_{g(t)}(B_{2r}) }{Vol_{g(t)}(B_r)}.
\end{align*}
Since $\frac{\partial}{\partial t} dVol=-R\cdot dVol$ along the Ricci flow, 
the volume ratio which appears in the last term can be estimated by
$$\frac{Vol_{g(t)}(B_{2r}) }{Vol_{g(t)}(B_r)} \leq e^{2KT}\frac{Vol_{g(0)}(B) }{Vol_{g(0)}(B')} \leq  e^{2KT}\cdot ce^K $$ 
where the last inequality comes from Bishop-Gromov volume comparison theorem.
Therefore, $\frac{1}{Vol_{g(t)}(B_r)}\|Rm\|_{L^p(B_r)}^p (t)\leq ce^{cKT}$ and we may apply Moser's iteration and derive a similar curvature bound as in \cite[pp. 2620-2623]{KotschwarMunteanuWang15}. 
\end{proof}

As one can observe easily, Theorem \ref{KMW} only involves the initial lower bound of $inj$, while Theorem \ref{MTLh}
involves $inj_{g(t)}$ for $t\neq 0$. 
On the other hand, the bounds in Theorem \ref{MTL} and \ref{MTLh} are much better than the one in Theorem \ref{KMW}. 

\begin{Rmk}
After checking the argument carefully, one can see that our geometrical blow-up method is valid for general geometric flow $\frac{\partial}{\partial t} g= \mathcal{R}$ possessing Shi's property, where $\mathcal{R}$ is a symmetric two-tensor defining by Ricci curvature and $g$. 
Indeed, if such a flow $\frac{\partial}{\partial t} g= \mathcal{R}$ satisfies $\frac{\partial}{\partial t} |Rm| \leq \Delta |Rm| + C |Rm|^2$, 
then its curvature can be controlled by Ricci curvature and the injectivity radius in the sense of Theorem \ref{MTL} and \ref{MTLh}. 
This shows that our approach is somewhat an alternative argument to Moser's iteration. 
Note that one more advantage of our approach is that one actually obtains a $C^{0,\sigma}$-bound, not only an $L^\infty$-bound. 
\end{Rmk}

\section{Applications}

\subsection{Compactness of the Ricci flow}
It has been known for decades that a sequence of closed connected Riemannian $n$-manifolds $\{(M_k,g_k)\}_{k\in\mathbb{N}}$ with bounded curvature, bounded diameter and volume bounded from below by a positive constant must contain a subsequence which converges in $C^{1,\alpha}$-topology to a Riemannian $n$-manifold $(M_\infty,g_\infty)$ (cf. \cite{GreeneWu88}).
Thus we say the collection 
$$\mbox{LV}:= \{  (M^n,g,L,V,D) \ |\ \  |Rm|\leq L, Vol\geq V \mbox{ and } diam\leq D\}$$ 
is pre-compact for any given real numbers $L\geq 0,V>0$ and $D>0$.
In \cite{Gao90}, L. Z. Gao showed that the same conclusion holds when the condition $|Rm|\leq L$ is replaced by $|Ric|\leq K$ and a certain integral bound of $|Rm|$. On the other hand, Anderson \cite{Anderson90} showed that, if $inj:=\inf_{M} inj(x)\geq I>0$, then $|Rm|\leq L$ can be replaced by merely $|Ric|\leq K$. That is, the set 
$$\mbox{KI}:= \{  (M^n,g,K,I,D)\ |\ \  |Ric|\leq K, inj\geq I \mbox{ and } diam\leq D \}$$ 
is also pre-compact for any given real numbers $K\geq 0,I>0$ and $D>0$. As observed by Cheng-Li-Yau \cite{ChengLiYau81}, and Cheeger-Gromov-Taylor \cite{CheegerGromovTaylor82} independently, the condition of injectivity radius can be derived from bounds of $|Rm|$ and $Vol$, thus LV-condition implies KI-condition. Note that the inverse is very likely not true although we do not notice any constructed counter-example in the limited literatures we have surveyed. In particular, one seems not able to improve the convergency from $C^{1,\alpha}$ to $C^2$ by using merely the KI-conditions.

Such convergence theory plays an important role in the study of singularities of the Ricci flow.
Indeed, given a singular portion of the flow, one can blow up the solution around it and characterize the singularity by using the limit of these rescaling solutions. Thus we need the compactness theorem derived by R. S. Hamilton in \cite{Hamilton95_2} to ensure the existence of such limiting solution. A particular version of Hamilton's theorem says that if a sequence of marked complete solutions of the Ricci flow $\{(M_k,g_k(t),x_k)_{t\in[0,T]}\}_{k\in\mathbb{N}}$ satisfies $|Rm|_{g_k}(x,t)\leq L$ for all $x,t$ and $Vol_{g_k(0)}(B_r(x_k))\geq V$ for some $r>0$, then there exists a subsequence converging in $C_{\rm{loc}}^{\infty}$-topology to a marked complete solution of the Ricci flow $(M_\infty,g_\infty(t),x_\infty)_{t\in(0,T]}$. 
The smooth convergency is due to Shi's estimate, which says that all higher order derivatives of $Rm$ are bounded provided that $Rm$ is bounded along the Ricci flow. Hence Hamilton's theorem can be seen as a LV-compactness theorem for the Ricci flow. 
By using Theorems \ref{MTLh} and \ref{KMW}, one can derive a $C^\infty$ KI-compactness theorem without assuming any bound on the curvature operator.  

\begin{Cor}\label{cpt}
Let $(M_k,g_k(t),x_k)_{t\in[0,T]}$ be a sequence of marked complete solutions of the Ricci flow. 
Suppose there are constants $K,\delta,I$ such that
$|Ric|_{g_k}\leq K \mbox{ on }M_k\times[0,T]$ and
$$\inf_{M_k\times[0,T]}inj\geq \delta \sqrt{t} \ \  \mbox{ or }\ \ 
\inf_{M_k\times\{0\}}inj\geq I>0\ \ \mbox{ for all } k,$$
then there exist a subsequence $S_j:=(M_j,g_j(t),x_j)$ and a solution $S_\infty:=(M_\infty,g_\infty(t),x_\infty)$ of the Ricci flow over $t\in(0,T]$ such that $S_j$ converges in $C^\infty$-topology to $S_\infty$ on every time interval $[\epsilon,T]$ with $\epsilon>0$ as $j\to\infty$.
\end{Cor}

\subsection{Ricci soliton}
Ricci solitons are manifolds $(M,g)$ coupled with a smooth vector field $X$, which can generate self-similar solutions to the Ricci flow. 
Indeed, if a Riemannian manifold $(M, g)$ satisfies $Ric_g+\frac{1}{2}L_Xg=\lambda g$ for some $X$, then $g(t)=\rho(t)\varphi_t^*g(0)$ solves the Ricci flow, where $\rho(t)=1-2\lambda t$ and $\varphi_t:M\to M$ is the one parameter family of diffeomorphisms generated by $\rho X$ (cf. \cite[Chapter 1]{ChowPart1}). Moreover, if $X=\nabla f$ for some smooth function $f:M\to\mathbb{R}$, then the soliton is called a gradient Ricci soliton.

Curvature growth is an important issue for the study of gradient Ricci solitons. Some classification results are built on the growth assumptions and, on the other hand, people expect that curvature of solitons should obey certain natural growth/decay laws. 
In \cite{MunteanuMTWang11}, Munteanu and M.-T. Wang proved that every shrinking gradient Ricci soliton with bounded Ricci curvature must have a polynomial bound of its curvature. The following theorem shows that, if the injectivity radius is bounded from below, then the curvature can be uniformly bounded by a constant.
 
\begin{Thm*bounded} 
Given any $\lambda \in \{ \pm\frac{1}{2}, 0\}$,
all $n$-dimensional Ricci solitons $Ric+\frac{1}{2}L_Xg=\lambda g$ with $|Ric|\leq K$ and $inj\geq I>0$ have the same curvature bound. 
\end{Thm*bounded}

\begin{proof}
Consider a self-similar solution generated by the soliton on a time interval $[0,t^*]$ for some $t^*>0$. Since the soliton changes only up to a scaling factor along the flow (modulo by diffeomorphisms), so the bounds of curvature and injectivity radius are changing according to the scaling factor. Indeed, $Ric_{g(t)}(x)=\rho^{-1}(t)Ric_{g(0)}(\varphi_t(x))$ and $inj_{g(t)}(x)=\rho^{\frac{1}{2}}(t)inj_{g(0)}(\varphi_t(x))$, where $\rho(t)=1-2\lambda t$, on the self-similar solution. 
By applying Theorem 2, we know that $|\nabla Ric|$ is uniformly bounded at $t=t^*$. That means $|Rm|\leq C$ at $t=t^*$, where $C$ depends on $K,I,t^*,\lambda$ and $n$. Because $g(0)$ differs to $g(t^*)$ only by a scaling factor, we have $|Rm|\leq C$ at $t=0$. 
\end{proof}

Based on this curvature estimate, one can further ask for compactness result.

\begin{Cor*GRS}
Let $\lambda\in \{\pm \frac{1}{2}, 0 \}$ and $A\in\mathbb{R}$. 
For any sequence of gradient Ricci solitons $(M_k,g_k,f_k,p_k)_{k\in\mathbb{N}}$ satisfying $Ric_{g_k}+Hess(f_k) = \lambda g_k$, if  
$$|Ric|_{g_k}\leq K,\ \ inj_{g_k}\geq I>0 \ \ \mbox{ and }\ \  |\nabla f|_{g_k}(p_k)\leq A,$$ 
then there exists a subsequence converging smoothly to $(M_\infty,g_\infty,f_\infty,p_\infty)$, which satisfies $Ric_{g_\infty}+Hess(f_\infty)= \lambda g_\infty$ with $f_\infty=\lim_{k\to\infty}f_k$.
\end{Cor*GRS}

\begin{proof}
As in the proof of previous corollary, we can evolve these solitons to some time $t^*>0$ and obtain uniform bounds for $Rm$ and injectivity radius for all $t\in[0,t^*]$. Moreover, all the derivatives of $Rm$ are bounded at $t=t^*$ by Shi's estimate. 
Hence, all the curvatures and their derivatives are bounded at $t=0$ and thus there exists a subsequential limit 
$(M_\infty,g_\infty,p_\infty)$. 

To show the convergence of $f_k$'s, we normalize $f$ by subtraction such that $f(p_k)=0$ and then use $|Hess(f_k)| = |\lambda g_k -Ric_{g_k}| \leq n|\lambda|+K$.
Indeed, the bound of $Hess(f_k)$ shows that $|\nabla f_k|$ grows at most linearly and $|f_k|$ grows at most quadratically from the point $p_k$, where $f_k=0$ and $|\nabla f_k|\leq A$. 
In particular, for any given $r<\infty$, there are uniform bounds for $f_k$'s and $|\nabla f_k|$'s on the geodesic balls $B_{r}(p_k)$. 
So $f_k$'s have a uniform $C^2$-bound for every fixed $r$. 
Moreover, all higher derivatives of $f_k$ depend only on the derivatives of $Ric$ and lower derivatives of $f_k$, thus are uniformly bounded.  
Therefore, there exists a subsequence of $(M_k,g_k,f_k,p_k)$ converging smoothly to a Ricci soliton $(M_\infty,g_\infty,f_\infty,p_\infty)$ satisfying $Ric_{g_\infty}+Hess(f_\infty)= \lambda g_\infty$.
 \end{proof}

\subsection{Curvature taming around the initial time}

In Theorem \ref{MTLh}, if we set $T=\frac{\eta}{K}$, then we obtain an estimate $|Rm|\leq C h^{-1}(t)$ where
$C$ is a constant depending only on $n,\delta$ and $\eta$. Surprisingly, this bound does not involve the bound of Ricci curvature. 
Combining with the fact that $inj_{g(0)}$ is not involved in Theorem \ref{MTLh}, we can prove that 
every Ricci flow is $1$-tamed by a universal constant up to a certain time.

\begin{Thm*taming}
There exists a universal constant $C=C(n)$ such that
for any smooth Ricci flow $(M^n,g(t))$ and any point $x_0\in M$, 
$$|Rm|(x,t)\leq Ct^{-1} \ \mbox{ on } B_{4r}(x_0,t)   \mbox{ for all } t\in (0,K^{-1}],$$ 
and
$$|Rm|(x,t)\leq CK \ \mbox{ on } B_{4r}(x_0,t)   \mbox{ for all } t\in [K^{-1},r^2],$$ 
where $r^2:= \inf \{ \ t>0\ |\  \inf_{B_{8r}(x_0,t)} inj < \sqrt{t} \}$ and $K$ denotes the maximum of $r^{-2}$
and 
$\sup |Ric|$ on $\bigcup_{[0,r^2]}B_{8r}(x_0,t)$.
\end{Thm*taming}

\begin{proof}
Since the flow is smooth, for any $x_0\in M$, there must exist an $r>0$ by the continuity of $inj$.
That is, $inj\geq \sqrt{t}$ on $\bigcup_{[0,r^2]}B_{8r}(x_0,t)$. 
By the definition of $K$, we have $|Ric|\leq K$ on $\bigcup_{[0,r^2]}B_{8r}(x_0,t)$.
Consider $y\in B_{4r}(x_0,0)$ and apply Theorem \ref{MTLh} with $\delta=\eta=1, T=\frac{1}{K}\leq r^2$ and $h=t$
on the domain $\bigcup_{[0,r^2]}B_{4r}(y,t)$, one obtains $|Rm|(y,t)\leq Ct^{-1}$ for some $C=C(n)$ and $t\in(0,K^{-1}]$.
Since $y$ is an arbitrary point in $B_{4r}(x_0,0)$, $|Rm|(x,t)\leq Ct^{-1}$ on $\bigcup_{(0,K^{-1}]}B_{4r}(x_0,t)$.
For $t\in[K^{-1},r^2]$, one may apply Theorem \ref{MTLh} on $\bigcup_{[t^*, t^*+K^{-1}]}B_{4r}(y,t)$ and 
see that $|Rm|(y,t^*)\leq CK$. Since $t^*$ is arbitrary, the second statement of the theorem is proved.
\end{proof}

\section{An estimate of $\nabla Ric$ without using injectivity radius}

For every Riemannian manifold, the traced second Bianchi identity
$\nabla_jR_{ij}=\frac{1}{2}\partial_iR$ holds.
In view of this, we say that a manifold satisfies {\it strong Bianchi inequality}
if the pointwise norm estimate $|\nabla Ric|\leq \beta|\nabla R|$ holds for some $\beta>0$.
For a solution of the Ricci flow, we consider a weaker condition as follows.
\begin{Def} Let $U$ be an open set of a manifold $M$. 
A solution of the Ricci flow $(U,g(t))_{t\in(0,T]}$ with $|Ric|\leq K$
is said to satisfy the weak Bianchi inequality 
if $$|\nabla Ric|\leq \alpha K t^{\frac{-1}{2}}+\beta|\nabla R|$$
on $U\times (0,T]$ for some constants $\alpha,\beta >0$.
\end{Def}

This inequality means that the trace-free part of $\nabla Ric$ is bounded
by either the traced part or a constant which is allowed to depend on $t^{\frac{-1}{2}}$.
This is not a strong restriction in the sense that it holds on a large class of static manifolds.
More details about this can be found in the next section.

In \cite{Shi89}, W.-X. Shi proved that if the curvature operator is bounded along
the Ricci flow, then all the derivatives of it are bounded uniformly except for the initial time.
If the boundedness condition of the full curvature operator is replaced by the one of Ricci curvature,
then it seems that Shi-type estimate does not hold.
However, if we impose the condition that weak Bianchi inequality holds, then we can 
derive a Shi-type estimate for Ricci curvature. 
Note that the weak Bianchi inequality is quite looser than the strong one,
because it allows $|\nabla Ric|\neq 0$ whenever $|\nabla R|=0$ at some point.

\begin{Thm*ShiG}[Global estimate]
There exists a constant $C>0$, depending only on $\alpha, \beta$ and $n$ such that
for every $n$-dimensional closed solution $(M^n,g(t))_{t\in[0,T)}$ of the Ricci flow,
if the Ricci curvature and its derivatives satisfy that
$|Ric|\leq K$ and $|\nabla Ric|\leq \alpha K t^{\frac{-1}{2}}+\beta|\nabla R|$
for all $t\in[0,\frac{1}{K}]\subset [0,T)$, where $K$ is a positive constant,
then $$|\nabla Ric|^2\leq C K^2t^{-1}$$ for all $t\in(0,\frac{1}{K}]$.
\end{Thm*ShiG}

\begin{Rmk}
When $t=0$, we define $\frac{1}{t}$ to be $\infty$. Hence the aforementioned inequalities, which are concerned, hold trivially.
\end{Rmk}

We recall the following evolution equations that will be used in the proof:
$$\frac{\partial}{\partial t}R=\Delta R+2|Ric|^2,\ \ \frac{\partial}{\partial t} R^2 =2R (\Delta R+2|Ric|^2) = \Delta R^2 -2|\nabla R|^2+4R\cdot|Ric|^2$$
and
\begin{eqnarray*}
\frac{\partial}{\partial t} |\nabla R|^2
&=& 2 \langle \nabla R, \nabla ( \Delta R+2|Ric|^2 ) \rangle  -2Ric(\nabla R,\nabla R) \\
&\leq& \Delta |\nabla R|^2 -2|\nabla^2R|^2 +4|Ric|\cdot|\nabla R|^2 +8 |Ric|\cdot|\nabla Ric|\cdot|\nabla R|.
\end{eqnarray*}

\begin{proof}

Since $\frac{\partial}{\partial t}R=\Delta R+2|Ric|^2$, by $|Ric|\leq K$, we have
$$\frac{\partial}{\partial t} R^2 \leq \Delta R^2 -2|\nabla R|^2+ CK^3,$$
where $C$ is an indefinite constant varying line by line.
Moreover, using $|\nabla Ric|\leq \alpha K t^{\frac{-1}{2}}+\beta |\nabla R|$, we can derive
$$\frac{\partial}{\partial t} |\nabla R|^2 \leq \Delta |\nabla R|^2 -2|\nabla^2 R|^2+ CK|\nabla R|^2 + CK^3t^{-1}.$$

Let $F=t|\nabla R|^2+AR^2$ for some constant $A$.
We can show that $\frac{\partial}{\partial t} F \leq \Delta F +C K^3$ whenever $A$ is
larger than some constant depending only on $\beta$.
Comparing with the o.d.e. $\frac{d}{dt} \phi(t)=CK^3$, one can prove that $F\leq C(K^2+K^3t)$ by maximum principle.
Hence $|\nabla R|^2\leq CK^2t^{-1}$. By using the weak Bianchi inequality again, we have $|\nabla Ric|^2\leq CK^2t^{-1}$.
\end{proof}

To show the local version, we have to do more efforts. 
To abbreviate the notation, we define the parabolic region emanated from $B_r(x_0,0)$ as
$$\mathcal{P}(r;x_0,t_0) := \Omega \times(0,t_0], \mbox{ where }  \Omega\subset M \mbox{ is the topological region defined by } B_r(x_0,0).$$

\begin{Thm*ShiL}[Local estimate]
There exist positive constants $\theta_0$ and $C$ depending only on $\alpha, \beta, n$ and $\Lambda$ such that
for every solution $(M^n,g(t))_{t\in[0,\theta_0/K]}$ of the Ricci flow,
if $|Rm|\leq \Lambda$ on $B_r(x_0,0)$,
$|Ric|\leq K$ and $|\nabla Ric|\leq \alpha K \left(\frac{1}{r^2}+\frac{1}{t}+K\right)^{\frac{1}{2}}+\beta|\nabla R|$
on $\overline{\mathcal{P}(r;x_0,t_0)}$ for some $r\leq\sqrt{\theta_0/K}$ and $t_0\leq\theta_0/K$, then
$$|\nabla Ric|^2\leq CK^2\left(\frac{1}{r^2}+\frac{1}{t}+K\right)$$ on $\mathcal{P}(\frac{r}{\sqrt{2}};x_0,t_0) $.
\end{Thm*ShiL}

\begin{proof}

Recall that along the Ricci flow
$$\frac{\partial}{\partial t} R^2 = \Delta R^2 -2|\nabla R|^2+4R\cdot |Ric|^2$$
and
$$ \frac{\partial}{\partial t} |\nabla R|^2
\leq \Delta |\nabla R|^2 -2|\nabla^2R|^2 +4|Ric|\cdot|\nabla R|^2+8|Ric|\cdot|\nabla Ric|\cdot|\nabla R|.$$
Denoting $u=\frac{1}{r^2}+\frac{1}{t}+K$, by the assumptions and Yang's inequality, we have
$$ \frac{\partial}{\partial t} |\nabla R|^2
\leq \Delta |\nabla R|^2 -2|\nabla^2R|^2 + C_1K|\nabla R|^2+C_1K^3u,$$
for some constant $C_1>0$.

Let $S=(BK^2+R^2)\cdot|\nabla R|^2$, where $B>\max\{n^2+4nC_1^{-1},32n^2\}$ is a constant.
We derive
\begin{eqnarray*}
\frac{\partial}{\partial t} S
&=& \frac{\partial}{\partial t}R^2\cdot|\nabla R|^2+ (BK^2+R^2)\frac{\partial}{\partial t}|\nabla R|^2\\
&\leq&  (\Delta R^2 -2|\nabla R|^2+ 4R\cdot |Ric|^2)\cdot |\nabla R|^2\\
&&        + (BK^2+R^2) (\Delta|\nabla R|^2 -2|\nabla^2R|^2 + C_1K|\nabla R|^2+C_1K^3u)  \\
&\leq& \Delta S - 2\nabla R^2\cdot \nabla |\nabla R|^2 -2|\nabla R|^4   -2(B+n^2)K^2|\nabla^2R |^2 \\
&&     + (C_1B+C_1n^2+4n)K^3|\nabla R|^2+ C_1(B+n^2)K^5u\\
&\leq& \Delta S - 2\nabla R^2\cdot \nabla |\nabla R|^2 -2|\nabla R|^4   -2BK^2|\nabla^2R |^2 \\
&&     + 2C_1BK^3|\nabla R|^2+ 2C_1BK^5u.
\end{eqnarray*}

We want to control the bad terms $2\nabla R^2\cdot \nabla|\nabla R|^2$, whose sign is unknown,
and $2C_1BK^3|\nabla R|^2$, which may not be bounded.
Indeed, using the following two inequalities, they can be absorbed by the other terms:
$$\left|2\nabla R^2\cdot \nabla|\nabla R|^2\right| \leq 8 nK|\nabla R|^2\cdot\left|\nabla^2 R\right|
\leq \frac{1}{2}|\nabla R|^4 + 32n^2K^2\left|\nabla^2R\right|^2$$
and $$2C_1BK^3|\nabla R|^2\leq  \frac{1}{2}|\nabla R|^4 +  2C_1^2B^2K^6 \leq \frac{1}{2}|\nabla R|^4
+  \frac{2}{3}C_1^2B^2K^5u .$$
Since $B>32n^2$, substituting these two inequalities into the evolution equation of $S$,
we get
$$
\frac{\partial}{\partial t} S
\leq \Delta S -|\nabla R|^4+C_2B^2K^5u
\leq \Delta S -\frac{S^2}{4B^2K^4}+C_2B^2K^5u,$$
for some constant $C_2$.
Consider $F=bSK^{-4}$ with some constant $b:= \min\{\frac{1}{4B^2},\frac{1}{C_2B^2}\}$ depending only on $n$,
one can derive $$\frac{\partial}{\partial t} F \leq \Delta F - F^2+ u^2. $$

To proceed the proof by using maximum principle, we need a space-time cut-off function.
However, the standard way to construct such a function requires the bound of $|\nabla Ric|$, 
which is exactly what we want to derive here. (Because the evolution equation of $|\nabla^2\varphi|$, which can be seen in the proof of Lemma B afterwards, involves $\nabla Ric$.)
This problem occurs also in the proof of Shi's estimate. 
To tackle this, Hamilton \cite[Section 13]{Hamilton95} used a continuity argument and eventually showed that there exists a short time $\theta_0K^{-1}$ such that Shi's estimate holds. 
The first step is to take a cut-off function $\varphi$ on the initial manifold $M$ satisfying
$\{ x\in M | \varphi>0 \}= B_r(x_0,0)$, $\varphi=r$ in $B_{\frac{r}{\sqrt{2}}}(x_0,0)$, $0\leq\varphi\leq r < Ar$,
$|\nabla \varphi|\leq A$ and $\left|\nabla^2\varphi\right|\leq \frac{A}{r}$ for some constant $A>1$ depending only on $n$ and the initial curvature bound $\Lambda$. 
Extend $\varphi$ to be a space-time function by letting $\varphi$ be independent of time. 
Since the Ricci flow is smooth, by continuity, $|\nabla \varphi|^2\leq 2A^2$ and $\varphi|\nabla^2\varphi|\leq 2A^2$
holds on $\mathcal{P}(r;x_0,\theta_1/K)$ up to some time $\theta_1/K>0$.
Moreover, we can construct a barrier function $H$ which behaves well up to $t=\theta_1/K$.  
\begin{Lma*A}
Let $H=\frac{cA^2}{\varphi^2}+\frac{d}{t}+K$ for some constants $c =14+4n$
and $d=2(1+\theta_1)$.
Then $\frac{\partial}{\partial t} H > \Delta H - H^2+u^2$ on $\mathcal{P}(r;x_0,\theta_1/K)$.
\end{Lma*A}
By using the maximum principle, one can show that $H-F$ cannot vanish on $\mathcal{P}(p,r,\theta_1/K)$.
Hence $H-F>0$ on $\mathcal{P}(p,r,\theta_2/K)$ for some $\theta_2>\theta_1$.
Combining with the following lemma, we can show that $\theta_1$ has a uniform lower bound,
i.e., $\theta_1$ must be larger than or equal to the uniform constant $\theta_0$ described in the following lemma.
We may assume $\theta_1<1$ (otherwise the uniform lower bound $\theta_0$ can be simply taken to be $1$). 
\begin{Lma*B}
There exists a constant $\theta_0$ which depends only on $\alpha,\beta,n$ and $\Lambda$ such that
if $|Ric|\leq K$ and $F\leq H$ on $\mathcal{P}(r;x_0,\theta/K)$ for some $\theta\leq\theta_0$ and $r\leq\sqrt{\theta/K}$,
then $|\nabla \varphi|^2\leq 2A^2$ and $\varphi|\nabla^2\varphi|\leq 2A^2$ on $\mathcal{P}(r;x_0,\theta/K)$.
\end{Lma*B}
Indeed, suppose on the contrary that $\theta_1<\theta_0$, then this lemma tells us that
the estimates of derivatives of $\varphi$ hold for time beyond $\theta_1$.
This contradicts the definition of $\theta_1$.

Therefore, $F<H$ on $\mathcal{P}(r;x_0,\theta_0/K)$.
We conclude that
$$|\nabla R|^2 = \frac{FK^4}{b(BK^2+R^2)}
\leq  \frac{K^4}{bBK^2}\left(\frac{(14+4n)A^2}{\varphi^2}+\frac{2(1+\theta_1)}{t}+K\right)
\leq  CK^2\left(\frac{1}{\varphi^2}+\frac{1}{t}+K\right)$$
and $$|\nabla Ric|^2\leq \alpha^2K^2u+\beta^2|\nabla R|^2
\leq \alpha^2K^2u+CK^2\left(\frac{1}{\varphi^2}+\frac{1}{t}+K\right)
\leq CK^2\left(\frac{1}{r^2}+\frac{1}{t}+K\right)$$
for some $C$ depending only on $\alpha,\beta, n$ and $\Lambda$.
\end{proof}

Now we prove Lemma A and Lemma B.

\begin{proof}[Proof of Lemma A]
We show that $-\frac{\partial}{\partial t} H + \Delta H + u^2 < H^2$ by the following calculations.
Using $|\nabla \varphi|^2\leq 2A^2$, $\varphi|\nabla^2\varphi|\leq 2A^2$
and $t\leq \theta_1/K$.
\begin{eqnarray*}
     -\frac{\partial}{\partial t} H + \Delta H + u^2
&=& \frac{d}{t^2}  + cA^2\Delta \left(\frac{1}{\varphi^2}\right) + \left(\frac{1}{r^2}+\frac{1}{t}+K\right)^2\\
&\leq& \frac{d}{t^2}  + \frac{cA^2}{\varphi^4}( 6|\nabla \varphi|^2 -2 \varphi\Delta\varphi )
                   + \left(\frac{1}{r^2}+\frac{1}{t}+\frac{\theta_1}{t}\right)^2\\
&\leq& \frac{d}{t^2}  + \frac{cA^2}{\varphi^4}( 12A^2 + 4nA^2 )
                   + 2\left(\frac{1}{r^2}\right)^2+2\left(\frac{1+\theta_1}{t}\right)^2\\
&\leq& \frac{(12+4n)cA^4}{\varphi^4} + 2\left(\frac{A^2}{\varphi^2}\right)^2 + \frac{2(1+\theta_1)^2+d}{t^2}\\
&=& \frac{((12+4n)c+2)A^4}{\varphi^4} + \frac{ 2(1+\theta_1)^2+d }{t^2}.
\end{eqnarray*}
Choose $c =14+4n $ and $d =2(1+ \theta_1)$,
then we have
$$-\frac{\partial}{\partial t} H + \Delta H + u^2   \leq \frac{((12+4n)c+2)A^4}{\varphi^4} + \frac{ 2(1+\theta_1)^2+d }{t^2}
\leq \left(\frac{cA^2}{\varphi^2}\right)^2 + \left(\frac{d}{t}\right)^2     \leq H^2   $$
\end{proof}

\begin{proof}[Proof of Lemma B]
By definition, $\nabla \varphi = g^{ij}\varphi_j e_i= \varphi^i e_i$. Thus
$$\frac{\partial}{\partial t} |\nabla \varphi|^2 =\frac{\partial}{\partial t}(g^{ij}\varphi_i\varphi_j)
=2R_{pq}g^{ip}g^{jq}\varphi_i\varphi_j\leq 2K|\nabla\varphi|^2$$
whenever $|Ric|\leq K$.
Therefore, $|\nabla \varphi|^2\leq A^2e^{2Kt}\leq 2A^2$ when $t\leq \frac{\theta}{K}$ and $\theta\leq \log\sqrt{2}$.

By using Uhlenbeck's orthonormal frame $\{E_a\}$ (cf. \cite[p. 155]{Hamilton86}), which satisfies
$\frac{\partial}{\partial t} E_a^i=g^{ij}R_{jk}E_a^k$,
one can derive
\begin{eqnarray*}
\frac{\partial}{\partial t} \nabla_a\nabla_b\varphi
&=& \frac{\partial}{\partial t} E_a E_b\varphi - \frac{\partial}{\partial t} (\Gamma_{ab}^c E_c\varphi) \\
&=& R_{bc} \nabla_a\nabla_c\varphi +R_{da} \nabla_d\nabla_b\varphi - (\nabla_aR_{cb}+\nabla_bR_{ac}-\nabla_cR_{ab})E_c\varphi.
\end{eqnarray*}
Hence $$
\frac{\partial}{\partial t} \varphi |\nabla^2\varphi|
=     \varphi \frac{\partial}{\partial t} |\nabla^2\varphi|
\leq  C\varphi(|Ric||\nabla^2\varphi| + |\nabla Ric||\nabla\varphi|).$$
By the assumption $F=b(BK^2+R^2)K^{-4}|\nabla R|^2\leq H=\frac{cA^2}{\varphi^2}+\frac{d}{t}+K$
and the weak Bianchi inequality,
we have $$|\nabla R|^2  \leq \frac{K^4}{b(BK^2+R^2)}\left(\frac{cA^2}{\varphi^2}+\frac{d}{t}+K\right)
                       \leq \frac{K^2}{bB}\left(\frac{cA^2}{\varphi^2}+\frac{d+\theta_1}{t}\right)$$
and $$|\nabla Ric| \leq \alpha K\sqrt{\frac{1}{r^2}+\frac{1}{t}+K}
                       +\beta \frac{K}{\sqrt{{bB}}}\sqrt{\frac{cA^2}{\varphi^2}+\frac{d+\theta_1}{t}}
                  \leq CK\left(\sqrt{\frac{cA^2}{\varphi^2}+\frac{d+\theta_1}{t}}\right),$$
where $C$ depends on $\alpha,\beta$ and $n$.
Recall that $\varphi\leq r$, $A>1$, $c=14+4n$ and $d=2(1+\theta_1)<4$.
Hence
$$
\frac{\partial}{\partial t} \varphi |\nabla^2\varphi|
\leq  CK\varphi|\nabla^2\varphi|
       +CK|\nabla\varphi|\sqrt{cA^2+\frac{(d+\theta_1)\varphi^2}{t}} \leq CK\left( \varphi|\nabla^2\varphi|+ A+\frac{r}{\sqrt{t}}\right),
$$
where $C$ depends on $\alpha,\beta$ and $n$.
By comparing with the ordinary differential equation $\frac{d}{d t} \phi = CK\left( \phi+ A+\frac{r}{\sqrt{t}} \right)$,
as Hamilton did in \cite[pp. 45-46]{Hamilton95}, one can show that
$$\varphi |\nabla^2\varphi|\leq e^{CKt}(A^2+CK(At+2r\sqrt{t})).$$
Therefore, when $r\leq \sqrt{\frac{\theta}{K}}$ and $t\leq \frac{\theta}{K}$ for some
$\theta=\theta(\alpha,\beta,n, A)$, we have $\varphi |\nabla^2\varphi|\leq 2A^2$.

\end{proof}

\section{Further discussions on Bianchi inequalities}

In this section, we discuss the validity of Bianchi inequalities on a fixed Riemannian manifold.
For general Riemannian manifolds, the derivative of Ricci tensor can be decomposed as follows.
\begin{Thm} [Cf. {\cite[p. 288]{Hamilton82}} for $n=3$]
Let $E_{ijk}=a(g_{ij}\partial_k R+ g_{ik}\partial_jR)+bg_{jk}\partial_iR$
with $a=\frac{n-2}{2n^2+2n-4}$ and $b=\frac{1}{2}-a(n+1)$.
Then the decomposition $\nabla_i R_{jk}=E_{ijk}+F_{ijk}$ satisfies that
$g^{ij}F_{ijk}=g^{jk}F_{ijk}=g^{ki}F_{ijk}=0$ and $\langle E_{ijk},F_{ijk} \rangle=0$.
In particular, we have 
$$|\nabla_iR_{jk}|^2= |E_{ijk}|^2+ |F_{ijk}|^2\ \ \  \mbox{ and }\ \ \ |E_{ijk}|^2= (a+b)|\nabla R|^2.$$
\end{Thm}

\begin{Rmk} 
When $n=3$, $a=\frac{1}{20}$, $b=\frac{3}{10}$ and $|E_{ijk}|^2=\frac{7}{20}|\nabla R|^2$;
when $n=4$, $a=\frac{1}{18}$, $b=\frac{2}{9}$ and $|E_{ijk}|^2=\frac{5}{18}|\nabla R|^2$.
\end{Rmk}

From this proposition, we know that a manifold satisfies the weak Bianchi inequality
if the trace-free part of $\nabla Ric$ can be bounded by the non-free part and a constant, 
i.e. $|F_{ijk}|^2 \leq |E_{ijk}|^2 + C$. 
Note that when $n\to\infty$, $a+b\to 0$ and thus $|\nabla Ric|\approx|F_{ijk}|$.

\begin{proof}
Using $g^{ij}F_{ijk}=g^{jk}F_{ijk}=0$ and the traced second Bianchi identity $\nabla_iR_{k}^{\ i}=\frac{1}{2}\partial_kR$, one can derive $(n+1)a+b=\frac{1}{2}$ and $2a+nb=1$. 
Thus $a=\frac{n-2}{2n^2+2n-4}$.

Furthermore, an easy computation shows that 
$|E_{ijk}|^2=\left(2(n+1)a^2+4ab+nb^2\right)|\nabla R|^2$ and 
$\langle E_{ijk},\nabla_iR_{jk}\rangle= (a+b)|\nabla R|^2 $.
Observing that $2(n+1)a^2+4ab+nb^2= a+b$, one obtains 
$\langle E_{ijk},F_{ijk} \rangle=\langle E_{ijk},\nabla_iR_{jk}-E_{ijk} \rangle=0$.
\end{proof}

On the other hand, we can compute explicitly on manifolds with rotationally symmetric metrics.

\begin{Thm}
Let $(M,g), g=dr^2+\varphi^2(r)g_{\mathbb{S}^{n-1}}$, be a rotationally symmetric $n$-dimensional manifold and $n\geq 3$.
Here $r$ is the arc-length parameter.
Denote the radial and spherical sectional curvatures as $K_0$ and $K_1$, respectively.
Suppose that
$\frac{\partial}{\partial r} K_0\cdot\frac{\partial}{\partial r}K_1\geq -\frac{C^2}{(n-1)^2-1}$
for some constant $C$. Then 
$$|\nabla Ric|^2\leq \left(\frac{1}{4}+\frac{1}{4(n-1)^2}\right) |\nabla R|^2  + (n-2)C^2$$ on $U$.
In particular, $(M,g)$ satisfies the strong Bianchi inequality $|\nabla Ric|\leq \frac{n}{2(n-1)}|\nabla R|$
whenever $\frac{\partial}{\partial r} K_0\cdot\frac{\partial}{\partial r}K_1 $ is nonnegative (e.g. a paraboloid or an infinite horn).
\end{Thm}

\begin{Rmk}
In this theorem, we do not assume that $|Ric|$ is bounded by some constant.
\end{Rmk}

\begin{proof}
It is well-known that for rotationally symmetric manifolds we have
$$Ric=(n-1)K_0 dr^2 + ( K_0+(n-2)K_1 )\varphi^2 g_{\mathbb{S}^{n-1}}$$
and $$R = (n-1)K_0 + (n-1)(K_0+(n-2)K_1).$$
Hence
$$|\nabla R|^2 
= 4(n-1)^2 \left(\frac{\partial}{\partial r}K_0\right)^2  
+ 4(n-1)^2(n-2)\left(\frac{\partial}{\partial r} K_0\right)\left(\frac{\partial}{\partial r}K_1 \right)
+(n-1)^2(n-2)^2 \left(\frac{\partial}{\partial r}K_1\right)^2  
$$
and
\begin{eqnarray*}
|\nabla Ric|^2 &=& (\nabla_1R_{11})^2+(\nabla_1R_{jj})^2\\
&=& (n-1)^2 \left(\frac{\partial}{\partial r} K_0 \right)^2
        +  \left(\frac{\partial}{\partial r} K_0+(n-2)\frac{\partial}{\partial r}K_1 \right)^2 \\
&\leq& \frac{(n-1)^2+1}{4(n-1)^2} |\nabla R|^2+ (n-2)\left( 1-(n-1)^2\right)\left(\frac{\partial}{\partial r} K_0\right)\left(\frac{\partial}{\partial r}K_1 \right)\\
&& +\frac{(n-2)^2}{4}\left(3-(n-1)^2\right)\left(\frac{\partial}{\partial r}K_1\right)^2\\
&\leq& \frac{(n-1)^2+1}{4(n-1)^2} |\nabla R|^2+ (n-2)C^2 .
\end{eqnarray*}
\end{proof}

It is interesting to know whether such Bianchi inequalities hold for generic solutions of the Ricci flow.
A related result which appeared earlier in a collaborated work of A. Deruelle and the author \cite[Theorem 2.10]{ChenDeruelle15}
shows that the validity of such inequalities may help us to resolve some long standing open problems about expanding gradient Ricci solitons.

\bibliography{refer_article}
\bibliographystyle{alpha}

\end{document}